\documentclass[11pt,a4paper]{article}
\usepackage{preambule}

\newcommand{\Pbf}{\mathbf{P}}
\newcommand{\bP}{\mathbf{P}}
\newcommand{\Ebf}{\mathbf{E}}
\newcommand{\bE}{\mathbf{E}}
\newcommand{\bbZ}{\mathbb{Z} }
\newcommand{\bbR}{\mathbb{R} }

\newcommand{\sign}{\mathrm{sign}}
\usepackage{xfrac}

\usepackage[backend=biber,datamodel=eprint-hal,maxbibnames=10,giveninits=true,style=alphabetic]{biblatex}
\title{Some properties of the principal Dirichlet eigenfunction\\ in Lipschitz domains, via probabilistic couplings}

\author{Quentin Berger and Nicolas Bouchot}
\date{}

\AtEndDocument{
\bigskip{\footnotesize%
	\textsc{LAGA, Universit\'e Sorbonne Paris Nord,  UMR 7539. 
99 Av. J-B Cl\'ement, 93430 Villetaneuse, France }\par 
\textsc{Institut Universitaire de France}\par
\textit{E-mail address}:
\texttt{quentin.berger@math.univ-paris13.fr} }\par
\bigskip{\footnotesize%
		\textsc{Institute f\"ur Mathematike, Innsbruck Universit\"at, Technikerstrasse 13, 6020 Innsbruck, Austria.} \par  
		\textit{E-mail address}: \texttt{nicolas.bouchot@uibk.ac.at} }}

\begin{document}
	
	\pagestyle{plain}
	
	\maketitle

\begin{abstract}
\noindent
We study a discrete and continuous version of the spectral Dirichlet problem in an open bounded connected set $\Omega\subset \mathbb{R}^d$, in dimension $d\geq 2$.
More precisely, consider the simple random walk on $\mathbb{Z}^d$ killed upon exiting the (large) bounded domain $\Omega_N = (N\Omega)\cap \mathbb{Z}^d$.
We let~$P_N$ its transition matrix and we study the properties of its ($L^2$-normalized) principal eigenvector~$\phi_N$, also known as ground state.
Under mild assumptions on~$\Omega$, we give regularity estimates on~$\phi_N$, namely on its $k$-th order differences (or \(k\)-th order derivatives), with a uniform control inside~$\Omega_N$.
We provide a completely probabilistic proof of these estimates: our starting point is a Feynman--Kac representation of~$\phi_N$, combined with gambler's ruin estimates and a new ``multi-mirror'' coupling, which may be of independent interest.
We also obtain the same type of estimates for the first eigenfunction $\varphi_1$ of the corresponding continuous spectral Dirichlet problem, in relation with a Brownian motion killed upon exiting~$\Omega$. 
Finally, we take the opportunity to review (and slightly extend) some of the literature on the $L^2$ and uniform convergence of $\phi_N$ to~$\varphi_1$ in Lipschitz bounded domains of~$\mathbb{R}^d$, which can be derived thanks to our estimates.


\end{abstract}

\section{Introduction}

The aim of this paper is to revisit a classical spectral Dirichlet eigenvalue problem in a bounded domain \(\Omega\), from a probabilistic perspective. 
Specifically, the principal Dirichlet eigenfunction, either in a discrete or continuous setting, arises as the limiting distribution of a random walk or a Brownian motion conditioned not to exit \(\Omega\); this is known as the quasi-stationary distribution (QSD) of the random walk or Brownian motion killed upon exiting the domain.
We refer to Section~\ref{sec:intro-rw} below for more details.

Our primary goal is to provide probabilistic proofs (primarily based on gambler's ruin estimates and coupling techniques) of regularity properties for the principal eigenfunction both at the discrete and at the continuous level --- in the spirit of difference or gradient estimates and Harnack inequalities, see~\cite[\S6.3]{lawlerRandomWalkModern2010}.
Most notably, our proofs is based on a Feynman--Kac representation for the principal eigenfunction (see e.g.\ \cite[\S6.3]{dynkin2002diffusions}), combined with a novel ``multi-mirror'' coupling to derive estimates on its $k$-th order differences or derivatives.
(We refer to Section~\ref{sec:coupling} where a brief overview of the coupling is given.)
In fact, we obtain upper bounds on its \(k\)-th derivative near the boundary of~\(\Omega\), which depend (somehow explicitly) on the shape (or regularity) of \(\partial \Omega\).
Let us acknowledge here that this is an extremely classical problem, but somehow surprisingly, the estimates we obtain on the principal Dirichlet eigenfunction appear to be new; we review some of the literature in Section~\ref{sec:comments} below.

The use of coupling arguments has already proven to be very helpful in the context of spectral theory, see \cite{atarburdzy2004,kendall1989} for the study of Neumann eigenfunctions, see also \cite{burdzykendall2000,chen1998} for some overviews of other possible applications.
Let us quote here Kendall~\cite{kendall1989} to highlight the interest of probabilistic techniques for spectral problems:
``Probability is useful here because it enables a coupling argument, working on individual sample paths. Analytical arguments via the heat equation tend to integrate the manifold variety of paths from point to point, and thus lose flexibility.''
To our knowledge, such techniques have not been used for Dirichlet eigenfunctions and their discrete approximations, and we aim at applying Kendall's philosophy in this context.

Let us also mention that the present paper initiated because we could not find a proper reference for the uniform convergence of the discrete eigenvector to its continuous counterpart; we derive this convergence as a corollary of our estimates.
This is also an extremely classical topic in numerical analysis and potential theory, but its literature is extensive and can be challenging to navigate for non-experts.
To help in this regard, we take here the opportunity to review (and complete) some of the literature on this convergence, see Section~\ref{sec:convergence} below.

\subsection{The Dirichlet eigenvalue problem and a discrete approximation}
\label{sec:Dirichlet}

\paragraph*{Spectral Dirichlet problem.}

Let $d\geq 2$ and let $\Omega$ be an open, bounded and connected set of $\mathbb{R}^d$; we assume for simplicity that it contains the origin.
Consider the spectral Dirichlet problem on $\Omega$, of finding $\mu\in \mathbb{R}\,,\; v\in L^2(\Omega)$ such that
\begin{equation}
\label{eq:Dirichlet}
\begin{cases}
-\Delta  v = \mu\, v & \quad \text{ on } \Omega \,,\\ 
\quad \ \ v=0 & \quad\text{ on } \partial \Omega \,,
\end{cases}
\end{equation}
with $\Delta = \displaystyle\sum_{i=1}^d \frac{\partial^2}{\partial x_i^2}$ the usual Laplace operator.

By spectral theory it is classical that the spectrum of the Dirichlet Laplacian is discrete with \textit{positive} eigenvalues; we refer for instance to \cite[Thm.~1.2.5]{henrot2006extremum}. We denote by $(\mu_k,\varphi_k)_{k\geq 1}$ the ordered eigenvalues and $L^2$-normalized eigenfunctions. 
Additionally, the smallest eigenvalue $\mu_1$ is simple with associated eigenfunction $\varphi_1$ of constant sign, see e.g.~\cite[Sec.~1.3.3]{henrot2006extremum} or more plainly~\cite[Thm.~6.34]{Borthwick2021}.

We will mostly consider the case of a \textit{Lipschitz domain} $\Omega$, \textit{i.e.}\  such that its boundary~$\partial \Omega$ is locally the graph of a Lipschitz function, \textit{i.e.}\ $x_{d}= \psi(x_1,\ldots, x_{d-1})$ (in appropriate local coordinates) for some Lipschitz function $\psi$.
This includes for instance smooth domains, with $\mathcal{C}^1$ boundary, it also allows for (reentrant) corners in the boundary, but not for cusp.
This Lipschitz domain condition is classical in the potential theory literature, and ensures for instance that there is a bounded trace operator $\gamma:H^{1}(\Omega) \to H^{1/2}(\partial \Omega)$, see~\cite[Th.~3.38]{mclean2000strongly}; in other words, we can give a pointwise meaning to $\varphi_k=0$ on $\partial \Omega$ (see our estimates or Proposition~\ref{prop:unifbound}).
Let us point to~\cite{agranovich2013spectral} for a review of spectral problems in domains that are Lipschitz, see in particular its Section~4 for the spectral Dirichlet eigenvalue problem.

\begin{remark}
	\label{rem:dirichlet}
	Let us stress that~\eqref{eq:Dirichlet} is referred to as the \emph{spectral (or eigenvalue) Dirichlet} problem.
	There are other types of Dirichlet (Poisson) problems, namely
	\[
	\begin{cases}
		\Delta u = f & \text{ in } \Omega \,,\\
		u=0 & \text{ in } \partial \Omega \,,
	\end{cases}
	\qquad \text{or} \qquad 
	\begin{cases}
		\Delta u = 0 & \text{ in } \Omega \,,\\
		u=g & \text{ in } \partial \Omega \,,
	\end{cases}
	\]
	for some given \(f,g\).
	All these problems are of course related, but the spectral Dirichlet problem~\eqref{eq:Dirichlet} is in fact non-linear, and can thus be considered as ``more difficult''. 
\end{remark}

\paragraph*{A discrete version of the problem.}

One can discretize the Dirichlet eigenvalue problem~\eqref{eq:Dirichlet} in several ways, and we focus on the finite-difference method, because of its clear relation with the simple random walk (see Section~\ref{sec:rw} below).

Let $h>0$ be a mesh size and consider $\mathbb{Z}_h^d\defeq (h \mathbb{Z}^d)$.
We then define the discretized versions of~$\Omega$ and~$\partial \Omega$: we let $\Omega^{(h)}$ be the connected component of $0$ in $\Omega \cap \bbZ_h^d$ (viewed as a subset of the lattice $\mathbb{Z}_h^d$) and 
$\partial \Omega^{(h)}  \defeq \{ x\in h\bbZ^d\setminus \Omega^{(h)}  \,, |x-y| <h \}$.
Let also $\Delta^{(h)}$ be the discrete analogue of the Laplacian~$\Delta$: for $v : \bbZ_h^d \longrightarrow \mathbb{R}$,
\[
\Delta^{(h)} v (x) = h^{-2} \sum_{e \in \{\pm e_i , 1\leq i \leq d\}} \big( v(x+h e) - v(x) \big) \,,
\]
where $e_i$ is the $i$-th vector of the canonical base of $\mathbb{Z}^d$.
Discrete potential theory results in Lipschitz domains exist, and let us give a few examples: 
for instance, \cite{varopoulos2001potential} provides kernel and Green function estimates (in particular near the boundary); 
\cite{varopoulos2009discrete} gives some (optimal) $L^{\infty}$ convergence for solution of the classical Dirichlet problem (\textit{i.e.}\ $\Delta u =0$ in~$\Omega$, with $u=f$ on the boundary); 
\cite{MS19} investigates the Martin boundary of unbounded globally Lipschitz domains.

Here, we consider the discrete analogue of the spectral Dirichlet problem~\eqref{eq:Dirichlet}: 
\begin{equation}
\label{eq:approxDirichlet}
\begin{cases}
-\Delta^{(h)}  v^{(h)} = \mu^{(h)}\, v^{(h)} & \quad \text{ on } \Omega^{(h)} \,,\\ 
\qquad \ \ v^{(h)}=0 & \quad\text{ on } \partial \Omega^{(h)} \,,
\end{cases}
\end{equation}
with $v^{(h)}: \Omega^{(h)}\cup \partial \Omega^{(h)} \to \mathbb{R}$.
We denote by $(\mu_k^{(h)},\varphi_k^{(h)})_{k\geq 1}$ the ordered eigenvalues and $L^2$-normalized eigenfunctions of~\eqref{eq:approxDirichlet} (see Remark~\ref{rem:normalization} below), and it is here easy to see (by the Perron--Frobenius theorem) that the first eigenvalue $\mu_1^{(h)}$ is simple with an associated eigenfunction~$\varphi_1^{(h)}$ which is positive in $\Omega^{(h)}$.

Then the finite difference method shows that, for ``smooth'' domains~$\Omega$, for any $k\geq 1$, $\lim_{h\downarrow 0} \mu_k^{(h)} = \mu_k$ and $\lim_{h\downarrow 0}\varphi_k^{(h)} = \varphi_k$ in $L^2$ and in the sup norm, with explicit rates. We refer to Section~\ref{sec:convergence} below for a more detailed discussion.

\begin{remark}[About the normalization]
	\label{rem:normalization}
	In this paper, we consider the $L^2$-normalized eigenfunctions, \textit{i.e.}\ such that 
	\[
	\|\varphi_k\|_{L^2}^2 \defeq \int_\Omega \varphi_k(x)^2 \dd x =1
	\quad \text{ and } \quad
	\|\varphi_k^{(h)}\|_{L^2,h}^2 \defeq h^{d}\sum_{x \in \Omega^{(h)}} \varphi_k^{(h)}(x)^2 =1 \,.
	\]
	One could also study $L^1$-normalized eigenfunctions $\phi_k$, $\phi_k^{(h)}$,  \textit{i.e.}\ such that $\int_\Omega |\phi_k(x)| \dd x =1$ and $h^{d}\sum_{x \in \Omega^{(h)}} |\phi_k^{(h)}(x)| =1$, or some $L^{\infty}$ or point-normalization, setting for instance the value at~$0$ to be equal to~$1$.
	All these choices turn out to be equivalent in our context, so we focus on the $L^2$-normalized version of the eigenfunction, see Remark~\ref{sec:normalization} below.
\end{remark}

\paragraph{Our main results in a nutshell.}

In the rest of the paper, we focus on the discrete principal eigenvalue $\mu_1^{(h)}$ and eigenfunction $\varphi_1^{(h)}$, and on their continuous counterpart.
Our goal is to obtain, via probabilistic techniques, the following estimates.

\begin{theorem}
\label{th:main}
Suppose that $\Omega$ is a Lipschitz domain (see Assumption~\ref{hyp:cone} below for a precise definition).
Then, there exist some (explicit) $p\in (0,1]$ that depends on the regularity of \(\partial \Omega\) (e.g.\ \(p=1\) under the stronger Assumption~\ref{hyp:D}), and some constants $c,C>0$ such that, for all $k\geq 0$:
	\[
	\begin{split}
	\Big| \frac{\partial^k}{\partial x_{i_1} \cdots \partial x_{i_k}}\varphi_1 (x) \Big| \leq c\, (C k)^k d(x,\partial \Omega)^{p-k} \,,& \qquad \forall\, x\in \Omega \,,\\
	\Big| D_{i_1, \cdots i_k} \varphi_1^{(h)} (x) \Big| \leq c (C k)^k d(x,\partial \Omega)^{p-k} \,,& \qquad \forall\, x\in \Omega^{(h)} \text{ with } d(x,\partial \Omega) > 2k h \,,
	\end{split}
	\]
where $D_{i_1,\ldots, i_k}$ is the $k$-th order finite difference operator in direction $(i_1, \ldots, i_k)$, see its definition~\eqref{derivatives} below.
Note that for $k=0$, the results are simply about the first eigenfunctions~$\varphi_1,\varphi_1^{(h)}$ (by convention we set \(0^0 =1\)).
\end{theorem}

These derivative or difference estimates should be classical results, but we were not able to find a reference in the literature, especially in the case of Lipschitz domains.
Let us simply mention that a notable feature of the bounds is that they control the eigenfunction and their derivatives near the boundary $\partial \Omega$. 
Comparing the discrete and continuous estimates, loosely speaking the results can be interpreted as telling that~$\varphi_1^{(h)}$ has the same regularity as its continuous counterpart~$\varphi_1$, \textit{uniformly in $h\in(0,1)$}.
Again, this is a classical and extremely studied subject, and below we comment further the literature and on the approximation of the continuous principal eigenfunction by its discrete counterpart.


\subsection{Simple random walk and Brownian motion in a (large) bounded domain}
\label{sec:intro-rw}

Our proofs rely on coupling techniques, so let us introduce the probabilistic objects that we use, and their relation to the Dirichlet spectral problems~\eqref{eq:Dirichlet}-\eqref{eq:approxDirichlet}.

\paragraph*{Simple random walk in $\Omega_N$}

The discretized eigenvalue problem~\eqref{eq:approxDirichlet} appears naturally in the context of random walks on~$\mathbb{Z}^d$, as follows.
Let $N\geq 1$ be a large integer and consider the  domain $\Omega_N \subset \bbZ^d$ defined as the connected component of $0$ in $(N\Omega) \cap \mathbb{Z}^d$.
In the case where the boundary $\partial \Omega$ is $\mathcal{C}^1$, then one can simply define 
\[
\Omega_N =  (N\Omega) \cap \mathbb{Z}^d \,,
\]
at least for $N$ large enough.
We also let $\partial \Omega_N \defeq \{x \in \mathbb{Z}^d  \setminus \Omega_N ,\, \exists y \in \Omega_N \text{ with } y\sim x\}$.
One can simply interpret $\Omega_N$ as $h^{-1} \Omega^{(h)}$ with $h=\frac1N$, but in the context of random walks we prefer to work with $\Omega_N \subset \bbZ^d$ rather than~$\Omega^{(h)}$.

We consider $\ZZ^d$ endowed with the Euclidean norm $| \cdot |$ and we denote by $d(\cdot, \cdot)$ its graph distance (given by the $1$-norm $|\cdot|_1$);
we denote $y\sim x$ if $x,y$ are nearest-neighbors, \textit{i.e.}\ $d(x,y)=1$.
We also denote by $\Delta_d$ the graph Laplacian on $\mathbb{Z}^d$, \textit{i.e.}\ $\Delta_d v (x) = \sum_{y\sim x} (v(y)-v(x))$.
Consider the matrix $P_N$ of the (nearest-neighbor) simple random walk on $\mathbb{Z}^d$ killed upon exiting $\Omega_N$, namely 
\[
P_N(x,y) = \begin{cases}
 \frac{1}{2d} \quad  & \text{ for } x\sim y, \, x,y\in \Omega_N, \\
 0 &\text{ otherwise}.
\end{cases}
\]
We then focus on the principal eigenvalue $\lambda_N=\lambda_{\Omega_N}$ of $P_N$ and its associated $L^2$-normalized eigenvector $\phi_N = \phi_{\Omega_N}$ (which is positive on~$\Omega_N$),
\begin{equation}\label{eq:def-phiN-L2}
P_N \phi_N = \lambda_N \phi_N \,, \qquad \text{ with }\ \|\phi_N\|_{L^2, \frac{1}{N}}^2 \defeq \frac{1}{N^d} \sum_{x\in \Omega_N} \phi_N(x)^2=1 \,.
\end{equation}
By definition of $P_N$, one easily verifies that $I_d-P_N= -\frac{1}{2d} \Delta_d$, so $\lambda_N, \phi_N$ are related to~\eqref{eq:approxDirichlet} with $h\defeq\frac1N$ in the following way: $\lambda_N \defeq 1- \frac{1}{2d}\mu_1^{(h)}$ and $\phi_N(x) = \varphi_1^{(h)}(h x)$.

One of our goal is to study properties of $\phi_N$ which can prove useful when considering the random walk \emph{conditioned to remain forever inside $\Omega_N$}. 
Indeed, one can introduce the Doob's $\phi_N$-transform of the simple random walk, defined by the transition kernel:
\begin{equation}\label{eq:def-p_N}
	\tilde p_N(x,y) \defeq \lambda_N^{-1} \frac{\phi_N(y)}{\phi_N(x)}\; \frac{1}{2d} \ind_{x \sim y} \qquad \text{ for } x,y\in \Omega_N \,.
\end{equation}
By standard Markov chain theory (see e.g. \cite[App.~A.4.1]{lawlerRandomWalkModern2010}),  the transition kernel~\eqref{eq:def-p_N} is the limit, as $t\to +\infty$, of the transition kernels of the SRW conditioned to stay in $\Omega_N$ until time~$t$ (see also~\eqref{eq:conditioning} below).
We thus refer to the Markov process with transition kernel~\eqref{eq:def-p_N} as the \textit{random walk confined in $\Omega_N$}, see Section~\ref{sec:cond-rw} for more details.

\begin{remark}
Let us notice that~\eqref{eq:def-p_N} can also be interpreted as the transition kernel of a random walk among conductances $c_N(x,y) \defeq \phi_N(x)\phi_N(y) \ind_{\{x\sim y\}}$, \textit{i.e.}\ we can rewrite 
\[
\tilde p_N(x,y) = \frac{c_N(x,y)}{\sum_{z\sim x} c_N(x,z)} \qquad \text{ for } x,y \in \Omega_N \,.
\]
Note also that the stationary distribution of the random walk confined in $\Omega_N$ is given by $\pi_N(x) =\phi_N^2(x)$, the $L^2$-normalization of $\phi_N$ making $\pi_N$ a probability distribution.
We refer to~\cite[App. A.4.1]{lawlerRandomWalkModern2010} for some details.
\end{remark}

\paragraph*{Principal eigenfunction as quasi-stationary distribution (QSD)}

Let us note that, for any $x \in \Omega$ and any bounded measurable $f : \Omega \longrightarrow \RR_+$, we have 
\begin{equation}
	\label{eq:QSD}
	\lim_{t \to +\infty} \Ebf_x \big[ f(X_t) \, \big| \, H_{\partial \Omega_N} > t \big] = \sum_{x\in \Omega_N} f(x) \hat\phi_N(x) \dd x  \, ,
\end{equation}
where \( \hat\phi_N(x) = \phi_N(x)/ \sum_y \phi_N(y)\) is the \(L^1\)-normalized principal Dirichlet eigenfunction.
Thus, \(\hat \phi_N\) is the large-time distribution of the simple random walk conditioned to remain inside~\(\Omega_N\).
One additionally has that \(\hat \phi_N\) is invariant for the conditioned process, \textit{i.e.}\ \(\Pbf_{\hat\phi_N}(X_t \in \cdot  \mid H_{\partial} > t) = \hat\phi_N (\cdot)\), so that \(\hat \phi_N\) is the \textit{quasi-stationary distribution} (QSD) of the random walk killed upon exiting \(\Omega_N\).
We refer to~\cite{collet2013quasi} for an overview on~QSD and to \cite[Sec.~7]{diaconis2020analytic} for how they relate to the geometry of the state space.

\begin{remark}
	Most articles on quasi-stationary distributions seem focused either on proving the convergence~\eqref{eq:QSD} (see~\cite{champagnat2023general} for a recent overview), on providing quantitative estimates on the convergence~\eqref{eq:QSD} (see e.g.\ \cite{diaconis2015quantitative}) or effective algorithms to estimate~\(\hat\varphi_1\) (see e.g.\ \cite{benaim2021stochastic,benaim2015stochastic}).
	Let us mention that there are also some relations with ergodic control problems, which may give another characterization of the QSD (see~\cite{BDNW21}).
	However, to our knowledge, there is no (obvious) way to derive our estimates from the existing QSD literature.
\end{remark}

\paragraph*{Brownian motion in $\Omega$}

Analogously, one can also consider the Brownian motion killed upon exiting $\Omega$ and the Brownian motion \emph{conditioned to remain forever inside $\Omega$}, see~\cite{pinsky1985convergence} or \cite[Ch.~2 and Ch.~5]{chung2012brownian}. 
Let us briefly present its relation with~\eqref{eq:Dirichlet}.

A Brownian motion killed when exiting $\Omega$ is a Markov process with transition kernel $\rho_{t}^{\Omega}(x,A) = \bP_x( X_t \in A,H_{\partial\Omega} >t)$, where $(X_s)_{s\geq 0}$ is a standard Brownian motion (starting from~$x$ under~$\bP_x$) and $H_{\partial\Omega} = \inf\{s>0, X_s \in \partial \Omega\}$ is the exit time of~$\Omega$.
Its generator is the Dirichlet Laplacian on~$\Omega$, \textit{i.e.}\ with boundary condition considered in~\eqref{eq:Dirichlet}.
Analogously to~\eqref{eq:def-p_N}, the Brownian motion \emph{conditioned to remain forever inside $\Omega$} is the Doob's $\varphi_1$-transform of the killed Brownian motion, where $\varphi_1$ is the first eigenvalue of the Dirichlet Laplacian, see~\eqref{eq:Dirichlet}.
In other words, this is the Markov process with transition density 
\begin{equation}
	\label{def:tildeBM}
	\tilde \rho_t(x,y) = \frac{\varphi_1(y)}{\varphi_1(x)} e^{\mu_1 t} \rho_t^{\Omega}(x,y) \,,	
\end{equation}
where $\rho_t^{\Omega}(x,y)$ is the density transition of the killed Brownian motion.

Additionally, we also have the large-time asymptotics analogous to~\eqref{eq:QSD}.
In other words, the \(L^1\)-normalized principal Dirichlet eigenfunction \(\hat\varphi_1 = \varphi_1/\|\varphi_1\|_1\) is the asymptotic distribution of the Brownian motion conditioned to remain in~\(\Omega\); it is also the quasi-stationary distribution of the Brownian motion killed upon exiting~$\Omega$.

\section{Main results: regularity properties of the eigenfunctions}
\label{sec:results}

Our main results give some properties of the first $L^2$-normalized discrete eigenvector~$\phi_N$, under some mild condition on the domain~$\Omega$; they have corresponding results for the continuous eigenfunction $\varphi_1$.
We will work with two different assumptions on the regularity of the boundary of $\Omega$, which affects the regularity of eigenfunctions near~$\partial \Omega$.

\paragraph*{Main assumptions on the domain}

A first (stronger) assumption is a ``uniform exterior ball'' condition, also known as positive reach, which informally tells that one can roll a ball on the outer boundary of $\Omega$. 
This includes the case of $\mathcal{C}^2$ boundaries, but also allows for non-reentrant corners.

\begin{assumption}[Uniform exterior ball condition/positive reach]
\label{hyp:D}
There exists some $\gep>0$ such that for any $x\in \partial \Omega$ there is some $z$ with $d(z,x)=\gep$ such that $B(z,\gep) \cap \Omega =\emptyset$.
\end{assumption}

Our second, weaker, condition is a ``uniform exterior cone'' condition; this is a standard assumption in the literature, corresponding to a uniform version on Poincaré's cone property. For bounded domains, this is actually equivalent to considering a Lipschitz domain, see \cite[Thm.~1.2.2.2]{grisvard2011elliptic}, but we prefer the exterior cone formulation since it appears naturally in our proofs (it has a probabilistic interpretation, see for instance in~\cite{deblassie1987}).

Let us first  state what we mean by exterior cone. 
For $z \in \RR^d\setminus \{0\}$, we write $\theta(z)$ the angle between $z$ and $e_1 =(1,0, \dots , 0)$. 
A right cone of angle $\alpha > 0$ is the open connected subset given by $\mathcal{C}_\alpha \defeq \big\{ z \in \RR^d \, : \, 0\leq \theta(z) < \alpha \big\}$; more generally, a cone originating at $x \in \RR^d$ with direction $y \in \mathbb{S}^{d-1}$ is the set $\mathcal{C}_{\alpha,y}(x) \defeq x + R_y (\mathcal{C}_\alpha)$ with $R_y$ the rotation of angle $\theta(y)$.
Now, given some positive $r > 0$ and $x \in \partial \Omega$, we define $\alpha_r(x)$ largest angle of a cone $\mathcal{C}$ with vertex $x$ which remains exterior to $\Omega$ for a distance $r$, as follows:
\begin{equation}
	\label{def:alphax}
	\alpha_r(x) \defeq \sup \big\{ \alpha \, : \, \exists y \in \mathbb{S}^{d-1} \text{ such that } \mathcal{C}_{\alpha,y}(y) \cap B(x,r) \cap \Omega =\emptyset \big\} \, .
\end{equation}
We can now state our exterior cone condition.

\begin{assumption}[Uniform exterior cone condition]
\label{hyp:cone}
There is some radius~$r=r_\Omega > 0$ such that 
\begin{equation}
	\alpha = \inf_{x \in \partial \Omega} \alpha_r(x)  > 0 \, .
\end{equation}
In other word, there is an $\alpha =\alpha_{\Omega} > 0$ such that, for all $x \in \Omega$, there exists an exterior open cone $\mathcal{C}_{\alpha}(x)$ with vertex $x$ and angle $\alpha$ such that $\mathcal{C}_{\alpha}(x) \cap B(x,r) \cap \Omega = \emptyset$.
\end{assumption}

These two assumptions are related to the regularity properties of the eigenfunctions, and will in particular impact how easy it is for a random walk or a Brownian motion to avoid exiting the domain $\Omega_N$ when near the boundary.

\subsection{Behavior of the principal eigenfunction near the boundary}

Let us now give our first results, which control the eigenfunction near the boundary.
In particular, they prove the case \(k=0\) in Theorem~\ref{th:main}.

\paragraph*{Under the uniform exterior ball condition, Assumption~\ref{hyp:D}}

Let us mention that our proof relies on simple gambler's ruin estimates (\textit{i.e.}\ probability to avoid a ball).

\begin{proposition}
\label{prop:unifbound}
Suppose that $\Omega$ satisfies Assumption~\ref{hyp:D}.
Then there is a constant $C>0$ such that:
\begin{enumerate}[label=(\roman*)]
	\item In the discrete setting: \(\displaystyle 0< \phi_N(x) \leq   C  \, \frac{d(x,\partial \Omega_N)}{N}\) for all \(x\in \Omega_N\).

	\item In the continuous setting: \(\displaystyle 0< \varphi_1(x) \leq   C  \, d(x,\partial \Omega)\) for all \(x\in \Omega\).
\end{enumerate}
In particular, we have $\sup_{x \in \Omega_N}  |\phi_N(x)| , \sup_{x \in \Omega}  |\varphi_1(x)|\leq C\, \mathrm{diam}(\Omega)$.
\end{proposition}

Let us stress that it is known that for smooth boundary \(\partial\Omega\), the continuous principal eigenfunction \(\varphi_1\) grows linearly with the distance, \textit{i.e.}\ \( \varphi_1(x) \geq c \, d(x,\partial \Omega)\), see for instance \cite[Thm.~7.1]{davies1984ultracontractivity}.
However, the principal eigenfunction may vanish faster near corners: for instance, for a \(d\)-dimensional hypercube, the eigenfunctions are simply products of one-dimensional eigenfunctions and therefore vanish as \(d(x,\partial \Omega)^d\) near corners.
We refer to \cite{grisvard2011elliptic} for an overview of the extensive literature estimating the boundary behavior in polygonal (or more general) domains. 
In particular, the behavior of \(\varphi_1(x)\) for small \(d(x,\partial \Omega)\) depends on the shape of the boundary near \(x\), but Proposition~\ref{prop:unifbound} somehow provides the best \textit{uniform} estimate one could hope for.

\begin{remark}
\label{sec:normalization}
Proposition~\ref{prop:unifbound} actually shows that the sup-norm of $\phi_N$ is controlled by its $L^2$-norm (equal to $1$), uniformly in $N \geq 1$.
Since the $L^2$-norm is obviously controlled by the sup-norm, we obtain that if $\Phi_N$ is the $L^{\infty}$-normalized principal eigenfunction, then we have that $\Phi_N = C_N \phi_N$ with a constant $C_N$ uniformly bounded away from $0$ and $\infty$.
Similar comparisons hold for the~$L^1$ norm.
\end{remark}

\paragraph*{Under the exterior cone condition, Assumption~\ref{hyp:cone}}

When using the exterior cone condition of Assumption~\ref{hyp:cone} rather than Assumption~\ref{hyp:D}, the main difference lies in the use of gambler's ruin estimates.
One needs to use estimates that a random walk or a Brownian motion avoids a cone of angle $\alpha$ rather than a ball; for this, \cite{banuelos1997brownian,deblassie1987} and \cite{DW15} turn out to be crucial.
Let us state the analogous of Proposition~\ref{prop:unifbound} in this context.

\begin{proposition}
\label{prop:unifcone}
Suppose that $\Omega$ satisfies Assumption~\ref{hyp:cone}.
Then there is a constant $C>0$ and some $p= p(\alpha) \in (0,1]$ defined in~\eqref{def:p} (or~\eqref{def:p2}) below, such that:
\begin{enumerate}[label=(\roman*)]
	\item In the discrete setting: \(\displaystyle 0< \phi_N(x) \leq   C  \, \Big( \frac{d(x,\partial \Omega_N)}{N} \Big)^{p}\) for all \(x\in \Omega_N\).
	
	\item In the continuous setting: \(\displaystyle 0< \varphi_1(x) \leq   C  \,  d(x,\partial \Omega)^{p}\) for all \(x\in \Omega\).
\end{enumerate}
In particular, we also have $\sup_{x \in \Omega_N}  |\phi_N(x)| , \sup_{x \in \Omega}  |\varphi_1(x)|\leq C\, \mathrm{diam}(\Omega)$.
\end{proposition}

The exponent $p = p(\alpha)$ depends on the minimal angle of the exterior cone in Assumption~\ref{hyp:cone}, and is related to the probability for a Brownian motion or a random walk to avoid a cone $\mathcal{C}_{\alpha}$ with opening angle $\alpha$ (note that this amounts to staying in a cone of opening angle $\theta= \pi-\alpha$); we refer to~\cite{banuelos1997brownian} for the Brownian motion case (see also~\cite{burkholder1977exit,deblassie1987}) and~\cite{DW15} for the random walk case.
The value of $p$ is in fact explicit for a generic cone $K$, see~\cite{banuelos1997brownian}: consider the Laplace--Beltrami operator on $\mathbb{S}^{d-1}$ and let $\lambda_{K^c}$ be its first Dirichlet eigenvalue in $K^c \cap \mathbb{S}^{d-1}$, then we have 
\begin{equation}
	\label{def:p}
	p= \sqrt{\lambda_{K^c} + \big(\tfrac d2 -1\big)^2} - \big(\tfrac d2 -1\big) \,.
\end{equation}
 
In the case of circular cones, there exist an alternative ``explicit'' expression for~$p$, see~\cite[p.192-193]{burkholder1977exit} or~\cite[\S2]{deblassie1987}.
Let $F(a,b;c;z) = \sum_{n=0}^{\infty} \frac{(a)_n (b)_n}{(c)_n} \frac{z^n}{n!}$ be the ordinary hypergeometric function, with the standard notation $(x)_n = x(x+1)\cdots (x+n)$ for the rising Pochhammer symbol, and define 
\[
h(p,\alpha) = F\big(-p,p+d-2;\tfrac12 (d-1); \cos^2(\tfrac{\alpha}{2}) \big) \,.
\] 
(In comparison to~\cite{burkholder1977exit,deblassie1987} we have set $\alpha =\pi-\theta$.)
Then, $p(\alpha)$ is the smallest positive zero of the function\footnote{In~\cite{burkholder1977exit,deblassie1987}, it is defined as the inverse map of $\alpha(p) = \sup\{ \alpha \in (0,\pi), h(p,\alpha)=0\}$, which is equivalent.} $p\mapsto h(p,\alpha)$, \textit{i.e.}
\begin{equation}
	\label{def:p2}
	p(\alpha) = \inf\{ p>0 , h(p,\alpha) =0 \} \,.
\end{equation}
In dimension $d=2$ we have the explicit expression $p(\alpha) = \frac{\pi}{2 (\pi-\alpha)}$; on the other hand, in dimension $d\geq 3$ we simply know that $\alpha \mapsto p(\alpha)$ is continuous and strictly increasing from $(0,\pi)$ onto $(0,+\infty)$, with $p(\pi/2) =1$.

\begin{remark}	
	Let us highlight the work \cite{vandenbergEstimatesDirichletEigenfunctions1999} where the authors obtain, in dimension $d=2$, the universal bound $|\varphi_1(x)| \leq c\, d(x,\partial \Omega)^{1/2}$ with a very mild condition on $\Omega$ (of uniform capacity density). 
	Our estimates actually improve this result under a stronger condition on \(\partial \Omega\).
	Indeed, under the uniform exterior cone Assumption~\ref{hyp:cone} with some angle~$\alpha$, we obtain that $|\varphi_1(x)| \leq c\, d(x,\partial \Omega)^{\frac12 + \kappa_\alpha}$ with $\kappa_\alpha = p(\alpha)-\frac12 = \tfrac{\alpha}{2(\pi-\alpha)}$ (note also that $\kappa_\alpha \to 0$ as $\alpha \to 0$).
\end{remark}

\begin{remark}
	\label{rem:px}
	We believe that the bounds in Proposition~\ref{prop:unifcone} are somehow the best \textit{uniform} estimate on \(\phi_N\) or \(\varphi_1\) one could hope for.
	Of course, as mentioned above, the eigenfunction will vanish faster near~\(\partial \Omega\) if the boundary is locally flat (or even has an outgoing corner).
	We could in fact improve Proposition~\ref{prop:unifcone} by letting the exponent $p$ in Proposition~\ref{prop:unifcone} depend on $x$.
	For $x\in \Omega_N$, we let $\partial \Omega_x = \{y\in \partial \Omega, d(x,y)=d(x,\partial \Omega)\}$ be the set of projected points of $x$ to $\partial \Omega$, and $\alpha_x = \inf_{y\in \partial \Omega_x} \alpha_r(y)$ be the least opening angle of cones associated to $\partial \Omega_x$ (note that this allows for outgoing corners if \(\alpha >\pi\)).
	Then, we could improve Proposition~\ref{prop:unifcone} as follows: for all \(x\in \Omega_N\)
	\begin{equation}
		\label{eq:bound-at-x}
		|\phi_N(x)| \leq   C  \, \Big( \frac{d(x,\partial \Omega_N)}{N} \Big)^{p_x}  \qquad \text{ with } p_x = p(\alpha_x) \,,
	\end{equation}
	where \(p(\alpha)\) is defined in~\eqref{def:p}-\eqref{def:p2} (and \(p(\alpha) >1\) for outgoing corners).
	(A similar statement would hold for \(\varphi_1\).)
	We have decided not to pursue this further for the simplicity of presentation.
\end{remark}

\subsection{Difference and higher-order difference estimates}
\label{sec:difference-estimates}

Let us now state our main results, that shows that $\phi_N$ varies regularly inside $\Omega_N$; we turn to the continuous setting afterwards.
For pedagogical purposes, we first give a control on differences of~$\phi_N$, and turn afterwards to higher-order differences (we need to introduce further notation); our proof follows the same pedagogical scheme, starting with a first-order difference estimate where the mirror coupling is simple.

\begin{theorem}[Differences of \(\phi_N\)]
\label{thm:reg-cone}
There is a constant $C>0$ such that, for any $z\in \Omega_N$ and any \(R \leq \frac12 d(z,\partial \Omega_N)\), we have
\begin{equation}
\label{eq:regcone}
\big| \phi_N(x)-\phi_N(y) \big| \leq  C \, \frac{d(x,y)}{R} \, \, \sup_{w\in B(z,R)} |\phi_N(w)|\,, \qquad  \forall\, x,y \in B\big(z,R) \,.
\end{equation}
\end{theorem}

As a consequence of Propositions~\ref{prop:unifbound} and~\ref{prop:unifcone}, choosing \(R = \frac12 d(z,\partial \Omega_N)\) in the above so that \(d(w,\partial\Omega_N) \leq 2 d(z,\partial \Omega_N)\) uniformly for \(w\in B(z,R)\), we get the following corollary.

\begin{corollary}
There is a constant \(C'\) such that, for any $z\in \Omega_N$, we have that
\begin{equation*}
\big| \phi_N(x)-\phi_N(y) \big| \leq C' \, \frac{d(x,y)}{N} \,  \Big( \frac{d(z,\partial \Omega_N)}{N}\Big)^{p-1} \,, \qquad  \forall\, x,y \in B\big(z,\tfrac12 d(z,\partial \Omega_N) \big)\,,
\end{equation*}
with \(p=1\) under Assumption~\ref{hyp:D} and $p \in (0,1]$ from~\eqref{def:p} (or~\eqref{def:p2}) under Assumption~\ref{hyp:cone}.	
\end{corollary}

Let us now define the $k$-th order difference of a function $\psi: \mathbb{Z}^d \to \mathbb{R}$.
For $i \in \{1,\ldots, d\}$, let~$D_i$ be the difference operator in direction $i$, defined as
\begin{equation}
	\label{def:Di}
	D_{i} \psi(x) = \frac{1}{2} \big( \psi(x+e_i)-\psi(x-e_i) \big) \,,
\end{equation}
with $e_i$ the $i$-th vector of the canonical basis.
(For instance, note that Theorem~\ref{thm:reg-cone} shows that $|D_i\phi_N (x)|\leq C/N$ uniformly in $x\in D_N$.)
Then, for $k\geq 2$, we define the $k$-th order difference
in directions $i_1,\ldots, i_k \in \{1,\ldots, d\}$ by setting 
\[
D_{i_1,\ldots, i_k}  =  D_{i_1} \cdots D_{i_k} \,.
\]
In fact, a simple iterative argument shows that, for any $i_1,\ldots, i_k \in \{1,\ldots, d\}$,
\begin{equation}
	\label{derivatives}
	D_{i_1,\ldots, i_k} \psi(x) = D_{i_1} \cdots D_{i_k} \psi(x) =  \frac{1}{2^k} \sum_{ \alpha \in  \{+1,-1\}^k} \sign(\alpha) \;  \psi\Big( x+ \sum_{j=1}^k \alpha_j e_{i_j} \Big) \,,
\end{equation}
where $\sign(\alpha) = (-1)^{m}$ with $m$ the number of ``$-1$'' in $\alpha$.

\begin{remark}
Let us stress that the $k$-th order derivatives are finite-difference schemes for the partial derivatives.
Indeed, for a function $\psi: \mathbb{R}^{d} \to \mathbb{R}$ and $h>0$, we can define 
\begin{equation}
	\label{def:Dh}
	D^{(h)}_{i_1,\ldots, i_k} \psi \defeq D_{i_1,\ldots, i_k} \psi_h \qquad
	 \text{ with }\ \psi_h : \Big| \begin{array}{cll} \bbZ^d & \to \bbR^d  \\ x &\mapsto h^{-1} \psi( hx )\end{array}  \,.
\end{equation}
Then, if $\psi$ is $k$ times differentiable at $x$, we have that $\frac{\partial^k}{\partial x_{i_1} \cdots \partial x_{i_k}} \psi (x) = \lim_{h\downarrow 0} D^{(h)}_{i_1,\ldots,i_k} \psi (x)$.
\end{remark}

We then have the following result on the higher-order differences of $\phi_N$.

\begin{theorem}[\(k\)-th order differences]
\label{thm:higherorder-cone}
There exists a constant $C>0$ such that, for any $k\geq 1$ and any $i_1,\ldots, i_k \in \{1,\ldots, d\}$, for any $x\in \Omega_N$ with $d(x,\partial \Omega_N)\geq 4k$ and any \( 2k \leq R \leq \frac12 d(x,\partial \Omega_N) \) we have
\begin{equation}
\label{eq:regcone2}
\big| D_{i_1,\ldots, i_k} \phi_N (x) \big| \leq  \Big( \frac{C k}{R} \Big)^{k}\, \sup_{w\in B(x,R)} |\phi(w)| \,.
\end{equation}
\end{theorem}

As a consequence of Propositions~\ref{prop:unifbound} and~\ref{prop:unifcone}, choosing again \(R = \frac12 d(x,\partial \Omega_N)\), we get the following corollary, which corresponds to the discrete part of Theorem~\ref{th:main}; note also that the case \(k=1\) recovers Theorem~\ref{thm:reg-cone}.

\begin{corollary}
	\label{cor:derivative}
There is a constant \(C'\) such that, for any $x\in \Omega_N$, we have that
\begin{equation*}
\big| D_{i_1,\ldots, i_k} \phi_N (x) \big| \leq \frac{(C' k)^k}{N^k} \, \Big( \frac{d(x,\partial \Omega_N)}{N} \Big)^{p-k} \,,
\end{equation*}
with \(p=1\) under Assumption~\ref{hyp:D} and $p \in (0,1]$ from~\eqref{def:p} (or~\eqref{def:p2}) under Assumption~\ref{hyp:cone}.
\end{corollary}

Let us note that the formulation~\eqref{eq:regcone} allows for some flexibility: in particular, if one has a better control on \(|\phi_N(z)|\) near the boundary (for instance if there is some corner), then the bound~\eqref{eq:regcone} improves.
In view of Remark~\ref{rem:px} (and in particular~\eqref{eq:bound-at-x}), one could indeed obtain the following: for any \(\gep \in (0,\frac12)\) choosing \(R = R_{\gep,x}= \gep\, d(x,\partial\Omega)\) in Theorem~\ref{thm:higherorder-cone}, we get that
\begin{equation}
	\label{eq:local-bound}
\big| D_{i_1,\ldots, i_k} \phi_N (x) \big| \leq \frac{(C' \gep^{-1} k)^k}{N^k} \, \Big( \frac{d(x,\partial \Omega_N)}{N} \Big)^{p_x^{(\gep)}- k} \,,
\qquad  p_{x}^{(\gep)} \defeq \inf_{w\in B(x,R_{\gep,x})} p_w\,,
\end{equation}
with the same notation as in Remark~\ref{rem:px}.
In particular, we could in theory make \(p_{x}^{(\gep)}\) arbitrarily close to \(p_x\) by choosing \(\gep\) small, at the cost of degrading the prefactor \(\gep^{-k}\).

\begin{remark}
	\label{rem:directional1} 
	We could also treat \emph{directional} $k$-th order differences, defining $D_{i^+}$, $D_{i^-}$ by  $D_{i^+} \psi (x) = \psi(x+e_i)-\psi(x)$ and $D_{i^-} \psi (x) = \psi(x)-\psi(x-e_i)$, and then $D_{i_1^{\gep_1},\ldots, i_k^{\gep_k}} = D_{i_1^{\gep_1}} \cdots D_{i_k^{\gep_k}}$ for $i_1,\ldots, i_k \in \{1,\ldots, d\}$ and $\gep_1,\ldots, \gep_k \in \{+1,-1\}$.
	For instance, we can write the graph Laplacian as $\Delta_d = \sum_{i=1}^d D_{i^+,i^-}$ 
	We chose to focus on the $k$-th differences defined in~\eqref{derivatives} for simplicity, and refer to Remark~\ref{rem:directional} below for a discussion on how to adapt the proofs (see in particular~\eqref{eq:directionaldiff}).
\end{remark}

We also prove the same results for the corresponding continuous first eigenfunction~$\varphi_1$. 
This does not seem completely standard for Lipschitz domains (we were not able to find a reference), but it derives easily from our coupling techniques.

\begin{theorem}[\(k\)-th order derivatives]
	\label{th:reg-cont-cone}
	There is a constant $C>0$ such that, for all $k\geq 1$ and any $i_1,\ldots, i_k \in \{1,\ldots, d\}^k$, for any $x\in \Omega$ with $d(x,\partial \Omega) \geq 4kh$ and any \(r \in (2kh, \frac12 d(x,\partial \Omega))\)
	\[
	\big| D^{(h)}_{i_1,\ldots, i_k} \varphi_1(x) \big| \leq  \Big(\frac{C k h}{r}\Big)^{k}  \sup_{w \in B(x,r)} |\varphi_1(w)|\,.
	\]
\end{theorem}

Note that, letting $h\downarrow 0$ then shows that for any $x\in \Omega$ and any \(r \in (0, \frac12 d(x,\partial \Omega))\), 
\[
	\Big| \frac{\partial^k}{\partial x_{i_1} \cdots \partial x_{i_k}}\varphi_1 (x) \Big| \leq \Big(\frac{C k}{r}\Big)^{k}  \sup_{w \in B(x,r)} |\varphi_1(w)| \,.
\]
As above, this upper bound depends on \(|\varphi_1(w)|\) in a small ball around \(x\), but a consequence of Propositions~\ref{prop:unifbound} and~\ref{prop:unifcone} is the following corollary.
It corresponds to the continuous part of Theorem~\ref{th:main}; note also that the case \(k=1\) recovers Theorem~\ref{thm:reg-cone}.

\begin{corollary}
	\label{cor:derivatives}
	There is a constant $C'$ such that, for all $k\geq 1$ and any $i_1,\ldots, i_k \in \{1,\ldots, d\}^k$, for all \(x\in \Omega\),
	\[
		\Big| \frac{\partial^k}{\partial x_{i_1} \cdots \partial x_{i_k}}\varphi_1 (x) \Big| \leq (C'k)^k d(x,\partial \Omega)^{p-k} \,,
	\]
	with \(p=1\) under Assumption~\ref{hyp:D} and $p \in (0,1]$ from~\eqref{def:p} (or~\eqref{def:p2}) under Assumption~\ref{hyp:cone}.
\end{corollary}

To conclude this section, let us mention that one could obtain some upper bound on the \(k\)-th order derivative that depends on the point \(x\in \Omega\), in the spirit of~\eqref{eq:local-bound}.

\subsection{Convergence of eigenfunctions \texorpdfstring{($L^2$ and $L^{\infty}$)}{} and some consequences}
\label{sec:convergence}

Recall that $(\mu_1,\varphi_1)$ are the principal eigenvalue and $L^2$-normalized eigenfunction of the Dirichlet Laplacian on~$\Omega$ (see~\eqref{eq:Dirichlet}). 
One can then consider the discretized function~$\varphi_N$ defined on $\Omega_N$ by
\[
\varphi_N : x\in \Omega_N \mapsto \varphi_1 \Big(\frac xN \Big)\,.
\]
This section is devoted to estimates on how close $\phi_N$ is to $\varphi_N$ in $L^2(\Omega_N)$ and $L^{\infty}(\Omega_N)$.
This is a classical topic, but we review here some of the literature and derive some of the following results from our difference estimates.

\smallskip

Before we turn to our proofs, let us note that we were not able to find a proper reference showing that having a Lipschitz boundary $\partial\Omega$ (\textit{i.e.}\ Assumption~\ref{hyp:cone}) is sufficient to ensure the $L^2$ or $L^{\infty}$ convergence of the principal discrete eigenfunction to its continuous counterpart.
The best result we found was that of~\cite[Cor.~7.1]{bramble1968effects} (see also~\cite{kuttler1970finite}), which show the convergence
\[
\lim_{N\to\infty} \frac{1}{N^d} \sum_{x\in \Omega_N} \big( \phi_N(x)-\varphi_N(x)\big)^2 = 0 \,.
\]
under the condition that the domain $\Omega$ has a smooth ($\mathcal{C}^2$) boundary, together with a rate of convergence.
Let us mention that~\cite{bramble1968effects} considers a finite-difference operator which is adjusted near the boundary, to obtain a $O(h^2)$ error term.
However, as observed in~\cite[Lemma~2.1]{bolthausen1994localization} and the discussion that follows, if one adapts the techniques of~\cite{bramble1968effects} to our setting, one obtains a $O(h)$ error.

In fact,~\cite[Thm.~6.1 and Thm.~7.2]{bramble1968effects} states a similar convergence result in dimension~$d=2$ if~$\Omega$ possesses reentrant corners (with a slower convergence rate).
We explain below in Appendix~\ref{app:Bramble} that the proof carries over to dimension $d\geq 2$, under the uniform exterior cone condition of Assumption~\ref{hyp:cone}.
Indeed, gambler's ruin estimates and bounds on $|D \varphi(x) |$ near the reentrant corners are central to \cite[Thm.~6.1]{bramble1968effects}, and the present paper provides all the necessary estimates to adapt it.

\begin{theorem}[$L^2$ convergence]
	\label{th:L2convergence}
	If the domain has a Lipschitz boundary, \textit{i.e.}\ if Assumption~\ref{hyp:cone} holds, then we have the~$L^2$ convergence for the principal eigenfunction: there is a constant $\kappa>0$ such that
	\begin{equation*}
		\frac{1}{N^d} \sum_{x\in \Omega_N}  \big(\phi_N(x)-\varphi_N(x)\big)^2 \leq \kappa N^{-p} \,,
	\end{equation*}
	where $p\in (0,1]$ is the exponent from~\eqref{def:p} (or~\eqref{def:p2}); $p=1$ if Assumption~\ref{hyp:D} holds.
\end{theorem}

Bramble and Hubbard also show the sup-norm convergence, and their method can also be adapted: we state the following result.

\begin{theorem}[$L^{\infty}$ convergence]
	\label{th:sup-convergence}
	Suppose that $\Omega$ satisfies Assumption~\ref{hyp:cone}, then
	\begin{equation*}
		\lim_{N\to\infty} \sup_{x \in \Omega_N} \big| \phi_N(x) - \varphi_N(x) \big| = 0 \, .
	\end{equation*}
	Moreover, if Assumption~\ref{hyp:D} holds, the convergence is at rate $N^{-1}$.
\end{theorem}

\begin{remark}
Of course these results have their counterparts in the notation of the discrete Dirichlet problem~\eqref{eq:approxDirichlet}, considering the convergence of $\varphi_1^{(h)}$ to $\varphi_1$ in $L^2(\Omega^{(h)})$ and $L^{\infty}(\Omega^{(h)})$, with $h\defeq \frac1N$.
Theorems~\ref{th:L2convergence} and~\ref{th:sup-convergence} translate into
\[
	h^d \sum_{x\in \Omega^{(h)}} \big| \varphi_1(x) -\varphi_1^{(h)}(x) \big|^2 \leq \kappa h^{p}
\qquad
\text{and}
\qquad
	\lim_{h \to 0} \sup_{x \in \Omega^{(h)}} \big| \varphi^{(h)}_1(x) - \varphi_1(x) \big| = 0 \, .
\]
\end{remark}

We refer to Appendix~\ref{app:Bramble} for a detailed discussion on the proof of~\cite{bramble1968effects}: we provide a summary of the method and how our results of Section~\ref{sec:results} can be used to obtain Theorems~\ref{th:L2convergence} and~\ref{th:sup-convergence}.
Let us mention that rate of convergence for the $L^{\infty}$ convergence in a Lipschitz domain is intricate, so we have preferred to stick to a simpler statement (again, see Appendix~\ref{app:Bramble} details).

\subsection{About the ratios of eigenfunctions}

One application of interest of Theorem~\ref{thm:reg-cone} is that it allows us to control the ratios of~$\phi_N$'s (resp.\ of $\varphi_1$'s).
Note that these ratios appear when considering the transition kernel of the confined random walk (resp.\ of the confined Brownian motion), see~\eqref{eq:def-p_N}; in particular, the ratios give the drift felt by the confined random walk. 

\begin{proposition}[Ratios of eigenfunctions]
	\label{prop:ratio-phiN}~
	\begin{enumerate}[label=(\roman*)]
		\item In the discrete setting: there is a constant $C > 0$ such that for any $z \in \Omega_N$ and any $R \leq \tfrac12 d(z,\partial \Omega_N)$, we have
		\begin{equation}
			\Big| \frac{\phi_N(x)}{\phi_N(y)} - 1 \Big| \leq C \frac{d(x,y)}{R} \frac{\sup_{w \in B(z,R)} \phi_N(w)}{\inf_{w \in B(z,R)} \phi_N(w)} \, , \quad \forall x,y \in B(z,R) \, .
		\end{equation}
		\item In the continuous setting: there is a constant $C' > 0$ such that for any $z \in \Omega$ and any $R \leq \tfrac12 d(z,\partial \Omega)$, we have
		\begin{equation}
			\label{eq:bound-ratio}
			\Big| \frac{\varphi_1(x)}{\varphi_1(y)} - 1 \Big| \leq C' \frac{d(x,y)}{R} \frac{\sup_{w \in B(z,R)} \varphi_1(w)}{\inf_{w \in B(z,R)} \varphi_1(w)} \, , \quad \forall x,y \in B(z,R) \, .
		\end{equation}
	\end{enumerate}
\end{proposition}

\begin{proof}
	We use Theorem \ref{thm:reg-cone} to get
	\begin{equation*}
		\Big| \frac{\phi_N(x)}{\phi_N(y)} - 1 \Big| \leq \frac{\big|\phi_N(x) - \phi_N(y) \big|}{\phi_N(y)} \leq \frac{1}{\phi_N(y)} \frac{d(x,y)}{R} \sup_{w \in B(z,R)} |\phi_N(w)| \, .
	\end{equation*}
	and take the infimum in the denominator. 
	The proof is identical in the continuous setting by considering Theorem~\ref{th:reg-cont-cone} instead.
\end{proof}

\begin{remark}[The case of a smooth boundary]
Let us highlight here the case of a smooth domain \(\Omega\) (say with a \(\mathcal{C}^{\infty}\) boundary). 
In view of Proposition~\ref{prop:unifbound} and its following comment, we have two constants \(c,C\) such that $c\, d(x,\partial \Omega) \leq \varphi_1(x) \leq C\, d(x,\partial \Omega)$ for all $x \in \Omega$. 
Therefore, if we fix $z \in \Omega$ and $R \leq \tfrac12 d(z,\partial \Omega)$, Proposition~\ref{prop:ratio-phiN} yields that \(|\frac{\varphi_1(x)}{\varphi_1(y)} - 1|\leq C''\,  \frac{d(x,y)}{R}\) uniformly for \(x,y \in B(z,R)\), since the ratio on the right-hand-side in~\eqref{eq:bound-ratio} is bounded.
This can actually be written in the following compact form:
\[
\Big| \frac{\varphi_1(x)}{\varphi_1(y)} - 1 \Big| \leq \hat C''\,  \frac{d(x,y)}{d(x,\partial \Omega) \wedge d(y,\partial \Omega)} \,.
\]
We stress that this shows that the ratios \(\varphi_1(x)/\varphi_1(y)\) are bounded from above by \(\frac32\) and from below by \(\frac12\) as soon that \(x,y\) are (much) closer between them than from the boundary.
Similar results hold in the discrete case.	
\end{remark}

\paragraph*{Some consequences in the bulk}

Theorems~\ref{th:L2convergence}-\ref{th:sup-convergence}, together with Proposition~\ref{prop:ratio-phiN}, allow us to derive a few corollaries of the properties of $\phi_N$ in the so-called \textit{bulk} of $\Omega_N$.
For any $\eta > 0$, define 
\[
\Omega_N^{\eta}\defeq \{x\in \Omega_N, d(x,\partial \Omega_N)  > \eta N\} \,,
\]
that we refer to as the bulk of $\Omega_N$.

\begin{corollary}\label{cor:cv-phiN-bulk-ratio}
	Assume that $\Omega$ satisfies the exterior cone condition of Assumption~\ref{hyp:cone}. 
	Then, for any $\eta>0$ we have
	\begin{equation}
		\sup_{x\in \Omega_N^{\eta}} \bigg| \frac{\phi_N(x)}{\varphi_N(x)} -1 \bigg| \xrightarrow{N\to\infty} 0\,.
	\end{equation}
\end{corollary}

\begin{proof}
Inside $\Omega$, since the first eigenfunction~$\varphi_1$ is continuous and positive, we get that for any $\eta>0$, there is some $\kappa_{\eta}>0$ such that $\varphi_1(x) \geq \kappa_\eta$ uniformly for $x\in \Omega$ with $d(x,\partial \Omega)>\eta$.
Then, from Theorem~\ref{th:sup-convergence}, we directly have that for any $x\in \Omega_N^{\eta}$,
\[
 \sup_{x\in \Omega_N^{\eta}} \bigg| \frac{\phi_N(x)}{\varphi_N(x) } -1 \bigg| \leq \frac{c}{\kappa_{\eta}}\sup_{x \in \Omega_N} \big| \phi_N(x) - \varphi_N(x) \big| \xrightarrow{N\to\infty} 0 \,,
\]
as desired.
\end{proof}

Another corollary is that, in the bulk, $\phi_N$ is bounded away from $0$ provided $N$ large enough. 
This is an immediate consequence of the positivity of $\varphi$ in the bulk of $\Omega$ (we have $\inf_{x\in \Omega_N^{\eta}} \varphi_N(x) \geq \kappa_{\eta}$ for some $\kappa_{\eta}>0$ independent of $N$) and Corollary \ref{cor:cv-phiN-bulk-ratio}. 

\begin{corollary}\label{cor:phiN-bounds}
	Assume that $\Omega$ satisfies Assumption~\ref{hyp:cone}. Then, for any $\eta > 0$, there is a constant $c_\eta > 0$ such that for all $N$ marge enough,
	\begin{equation}
		c_\eta \leq \inf_{x\in \Omega_N^{\eta}} \phi_N(x) \leq \sup_{x\in \Omega_N^{\eta}} \phi_N(x) \leq \frac{1}{c_\eta} \, .
	\end{equation}
	In particular, in the bulk~$\Omega_N^{\eta}$, the ratios $\phi_N(x)/\phi_N(y)$ are uniformly bounded away from $0$ and~$\infty$.
\end{corollary}

Let us also highlight that we can in fact control the ratios \(\phi_N(x)/\phi_N(y)\) when \(x,y\) are close to each other.
Indeed, Proposition~\ref{prop:ratio-phiN} and Corollary~\ref{cor:phiN-bounds} imply that, uniformly for $x,y\in \Omega_N^{\eta}$ with $x \sim y$, we have
\[
 1- \frac{C_{\eta}}{\kappa_{\eta}} N^{-1} \leq \frac{\phi_N(x)}{\phi_N(y)}  \leq  1+ \frac{C_{\eta}}{\kappa_{\eta}} N^{-1} \,.
\] 
Bounding \(1- \frac{C_{\eta}}{\kappa_{\eta}} N^{-1} \geq e^{-c_{\eta} N^{-1}}\) and \(1+ \frac{C_{\eta}}{\kappa_{\eta}} N^{-1} \leq e^{c_{\eta} N^{-1}}\), we get the following corollary, thanks to a telescopic product.

\begin{corollary}
	Assume that $\Omega$ satisfies Assumption~\ref{hyp:cone}. Then, for any $\eta > 0$, there is a constant \(c_{\eta}>0\) such that, for all \(N\) large enough,
	\[
	\forall \, x,y\in \Omega_N^{\eta} \qquad \exp\Big( - c_{\eta} \frac{d(x,y)}{N} \Big) \leq \frac{\phi_N(x)}{\phi_N(y)} \leq \exp\Big( c_{\eta} \frac{d(x,y)}{N} \Big) \, .
	\]
\end{corollary}

\section{Some comments about the proof and the literature}

Let us now comment on one of the key ingredient that we use, namely so-called Feynman--Kac relations, that allows us to develop coupling techniques.
We discuss further the literature afterwards.

\subsection{Feynman--Kac relations}
\label{sec:Feynman}

As far as the simple random walk is concerned, let us denote \(H_{\partial \Omega_n} \defeq \inf\{ n\geq 0, X_n \in \partial \Omega_N\}\) the exit time of \(\Omega_N\), and let $\tilde \bP_x^N$ be the law of the \textit{confined} random walk when started from~$x$, \textit{i.e.}\ the Markov chain on $\Omega_N$ with transition probabilities $\tilde p_N(x,y)$ from~\eqref{eq:def-p_N}.
Then, using the form~\eqref{eq:def-p_N}, after telescoping the ratios of the $\phi_N$'s, we obtain that for any stopping time~\(\tau\) and any event \(A\in \mathcal{F}_{\tau}\),
\begin{equation}
	\label{eq:PtildeP}
	\tilde \bP_x^N(A) = \frac{1}{\phi_N(x)} \bE_x\Big[ (\lambda_N)^{-\tau} \phi_N(X_\tau) \ind_{  \{ \tau < H_{\partial \Omega_N} \}}  \, \ind_{A} \Big] \,.
\end{equation}
This is in fact a Feynman--Kac type relation for the first eigenfunction, see e.g.\ \cite[\S6.3]{dynkin2002diffusions}, which expresses the law of the confined random walk in terms of a \(\phi_N\)-transform of the original random walk.

A similar Feynman--Kac relation holds for the Brownian motion conditioned to stay forever in~$\Omega$.
Let $\tilde \bP_x$ denote the law of a Brownian motion $(X_s)_{s\geq 0}$ starting from $x$ and conditioned to remain forever in~$\Omega$. 
By using the Doob's $\varphi_1$-transform of~\eqref{def:tildeBM}, if $\tau$ is a stopping time and $A \in \mathcal{F}_{\tau}$, we have
\begin{equation}
	\label{eq:PtildePBM}
	\tilde \bP_x(A) = \frac{1}{\varphi_1(x)} \bE_x \Big[  e^{\mu_1 \tau} \varphi_1(X_{\tau}) \ind_{\{\tau < H_{\partial \Omega}\}} \ind_A\Big]\,.
\end{equation}

As a consequence, if we apply~\eqref{eq:PtildeP} to some stopping time that verifies \(\tau < H_{\partial \Omega_N}\) for any random walk trajectory starting from \(x\) (or~\eqref{eq:PtildePBM} with some \(\tau\) such that \(\tau < H_{\partial \Omega}\) for any Brownian trajectory starting from \(x\)) and take \(A\) the full event, we get that
\begin{equation}
	\label{eq:Feynman}
	\phi_N(x) = \bE_x\Big[ (\lambda_N)^{-\tau} \phi_N(X_\tau) \Big]   
	\qquad \text{and} \qquad 
	\varphi_1(x) = \bE_x \Big[  e^{\mu_1 \tau} \varphi_1(X_{\tau}) \Big] \,,
\end{equation}
where \(X\) is either a simple random walk (on the left) or a Brownian motion (on the right).

\subsection{About the coupling techniques}
\label{sec:coupling}

The idea behind the proof of Theorem~\ref{th:main} is to use of the formula~\eqref{eq:Feynman} to estimate differences \(\phi_N(x) -\phi_N(y)\) (or \(\varphi_1(x)-\varphi_1(y)\)).
In a few words, we take \(\tau\) to be the hitting time of a ball contained in \(\Omega\) that contains both \(x,y\), so that we can use~\eqref{eq:Feynman} for two random walks with starting points \(x,y\), \textit{with the same stopping time}.
Then, we couple two such random walks \(X^{(1)},X^{(2)}\) by the (standard) \emph{mirror coupling}, see~\eqref{def:mirror} below for a proper definition.
In particular, if the coupling succeeds before exiting the ball, when using~\eqref{eq:Feynman} the terms \((\lambda_N)^{-\tau} \phi_N(X_\tau^{(i)})\) for \(i=1,2\) cancel out.
One is therefore reduced to estimating the probability that the coupling does not succeed before exiting the ball, and this is done by using gambler's ruin estimates (with the slight subtle fact that one needs to control \(\phi(X_{\tau})\) and the growing term \((\lambda_N)^{-\tau}\)).
We refer to Section~\ref{sec:regularity}, where this strategy is developed and explained in details.

As far as \(k\)-th order differences are concerned, we use a similar idea, but one now needs to couple \(2^k\) random walks with starting points \(x + \sum_{j=1}^k \pm e_{i_j} \) (see~\eqref{derivatives} for a definition of the \(k\)-th order difference).
The idea is to consider \(k\) independent random walks \((S^{(j)})_{1\leq j \leq k}\) and their mirror versions \((\hat S^{(j)})_{1\leq j \leq k}\), from which we construct \(2^k\) random walks by summing either \(S^{(j)}\) or \(\hat S^{(j)}\) for \(1\leq j\leq k\).
One can then check that the starting points of the \(2^k\) random walks are indeed the correct ones, and that the coupling succeed as soon as \textit{one of the single mirror couplings} (\textit{i.e.}\ for \(S^{(j)}\)) succeeds.
We refer to Section~\ref{sec:BM-couplings} for details (in the context of Brownian motions, where the argument is more transparent).
In particular, the probability that the multi-mirror coupling fails is the \(k\)-th power of the probability that one single mirror coupling fails, which reduces again to gambler's ruin estimates.
Note here that having reduced the coupling of \(2^k\) walks to that of \(k\) walks is a crucial step, and explains the \(k\)-th power appearing in Theorems~\ref{thm:higherorder-cone} and~\ref{th:reg-cont-cone}.

\paragraph*{On the coupling techniques in the literature}

In addition to giving important information on the first eigenvector~$\phi_N$ (\textit{resp.}~$\varphi_1$) that are useful on their own, we believe that the relatively simple probabilistic ideas behind their proofs might have applications in other contexts.
We should point out that coupling techniques have already been used in the context of spectral problems, see \cite{atarburdzy2004,burdzy2006,burdzykendall2000,chen1998,kendall1989} to cite a few. 
Let us mention in particular \cite{atarburdzy2004}, which establishes via a coupling argument the ``hot spots conjecture'' that the Neumann eigenfunction associated with the second Neumann eigenvalue attain its maximum and minimum at boundary points only.
In fact, the process associated with the Neumann Laplacian is the Brownian motion \emph{reflected at the boundary of $\Omega$} (as opposed to the Brownian motion killed at the boundary for the Dirichlet Laplacian), and~\cite{atarburdzy2004} use a type of mirror coupling that needs to apply also when the Brownian motion hits the boundary.
In comparison, our mirror coupling is much more elementary.

We also mention Kendall's coupling \cite{kendallNonnegativeRicciCurvature1986} and its refinement by Cranston \cite{cranstonGradientEstimatesManifolds1991} which provide control on the gradient of \(\varphi_1\), namely on $\| \nabla \varphi_1 \|_\infty / \| \varphi_1 \|_\infty$ \cite{arnaudonGradientEstimatesDirichlet2020,wangGradientEstimatesDirichlet2004} (in analogy with our derivative estimates in Theorem~\ref{thm:higherorder-cone} below) as well as on $\lambda_1$ \cite{chenApplicationCouplingMethod1994,chenGeneralFormulaLower1997,wangApplicationCouplingMethods1994}. 
The main idea here is to couple two diffusion processes so that they follow some mirrored SDE (that is two SDEs whose coefficients are ``mirrors'' from each other) until they meet. 
Kendall's coupling is a bit more involved than our approach, as it deals with Riemannian manifolds (and can even be extended to more complex settings, see \textit{e.g.}\ \cite{baudoinNoteFirstEigenvalue2022}).

Let us conclude by saying that, to the best of our knowledge, coupling techniques have only been used in continuous settings (and not for the spectral Dirichlet problem), but never in the context of discrete approximations, as in the present paper. We hope that our results will highlight some of the probabilistic ideas at hand and prove valuable in other contexts.

\subsection{Further comments}
\label{sec:comments}

\paragraph*{On controlling the eigenfunction near the boundary}

The estimates of solutions of PDEs near the boundary of a domain is a huge area in analysis.
A wide range of papers prove $L^p$ controls of the solution at the boundary, through Hardy-type inequalities, which usually give a ``weighted'' control on the behavior near the boundary; we refer to~\cite{Davies1999} for an overview. 
An famous example is \cite{frommThirdDerivativeEstimates1994} where, the authors prove, for convex $\Omega$, the integrability of $(\nabla^3 u)(x) d(x,\partial \Omega)^\eps$ for any $\eps > 0$, where $u$ is the solution of a Dirichlet problem (\(\Delta u = f \in \mathcal{C}^{\infty}(\bar \Omega)\) in \(\Omega\) and \(u=0\) on \(\partial \Omega\), recall Remark~\ref{rem:dirichlet}).
However, estimates in cases where derivatives or higher-order derivatives may blow up near the boundary, as in Theorem~\ref{th:reg-cont-cone}, seem to be rarely addressed.

\paragraph{Gambler's ruin estimates}

Quasi-stationary distributions of Markov chains can also help obtain information on gambler's ruin estimates. We notably highlight \cite[Sec. 6.5]{diaconisGamblersRuinEstimates2021} which explains how a proper control on $\phi_N$ yields, under suitable conditions on $\Omega$, numerous estimates for the walk absorbed upon leaving $\Omega_N$. This so-called ``inner-uniform'' condition (see \cite[Def. 6.1]{diaconisGamblersRuinEstimates2021}) is automatically verified if $\Omega$ satisfies an interior version of our own Assumptions \ref{hyp:D} (uniform interior ball) or \ref{hyp:cone} (uniform interior cone).

As an example, the result of \cite[Thm. 5.11]{diaconisGamblersRuinEstimates2021} on the Poisson kernel of SRW, combined with our own estimates, shows that
\[
\Pbf_x \big( X_{H_{\partial \Omega_N}} = y \big) \leq c N^d (1 - \lambda_N)^{-1} \phi_N(x) \sum_{z \sim y} \phi_N(y) \leq c_\eta N^{p-d} \,,
\]
for some $c_\eta$ that is uniform in $x \in \Omega_N^\eta$ and $y \in \Omega_N$. When $\Omega_N$ is ball-like, this is a classical result, see \cite[Lem. 6.3.7]{lawlerRandomWalkModern2010}.



\paragraph{On the eigenfunction approximation}

Let us mention that another classical way of approximating eigenvalue problems is the finite element method. This method also yields the uniform convergence of a discrete problem towards the continuous one, see e.g.\ \cite{ciarlet1973maximum,schatz1977interior} or~\cite{ciarlet2002finite} for an overview, however some major differences arise. 
First, the discretization is done on regular triangulations which are optimized to get a faster convergence rate. Secondly, the discrete functions are in fact transported back to the continuous setting by interpolation over the triangulation. 
This allows the study of the eigenvalue problem on this particular class of functions while still keeping the original Laplace operator on $\Omega$.

The finite difference method instead fully discretizes the problem and forgets about the continuous setting. 
This implies to work with discrete functions as well as a discrete Laplace operator, which is much closer to an analysis using Markov chains such as the random walk.
Let us stress that the proof of Bramble and Hubbard \cite{bramble1968effects} relies on an analysis of the Green's function of the random walk killed at the boundary, as well as a Walsch approximation theorem. 
In comparison, our paper gives some  uniform control of the discrete eigenvector $\phi_N$.

\paragraph*{A word on some of our motivations}

One of our main motivation for this work was to study the random walk confined in $\Omega_N$ (or its quasi-stationary distribution), whose transition kernel is given by~\eqref{eq:def-p_N} and involves rations of the form $\phi_N(x)/\phi_N(y)$.
Another interpretation is that the confined random walk is a random walk in conductances $c_{N}(x,y) = \phi_N(x)\phi_N(y)$, referred to as \textit{tilted} random walk, so our regularity results appear crucial in studying fine properties of this process.

For instance, the regularity of $\phi_N$ can be used to estimate on which time scale the tilted random walk is comparable to the simple random walk, \textit{i.e.}\ with constant conductances --- or with the random walk with conductances $\tilde c_{N}(x,y) = \phi_N(x)\phi_N(y)$, using convergence results.
As another example, the second author investigates in~\cite{bouchot2024confined} the geometry of the \textit{confined} random walk (\textit{i.e.}\ conditioned to remain forever in $\Omega_N$), in the bulk of~$\Omega_N$, giving a coupling between the confined walk and a \textit{tilted} random interlacement among conductances \(\tilde c_N(x,y)\).
Here again, the regularity of $\phi_N$ could be useful to determine on which scales the \textit{tilted} random interlacements are comparable with standard random interlacements.
The results of the present article also appear crucial to study covering times of the confined random walk, which is the object of~\cite{bouchot-covering}.

\section{Some preliminaries and estimates near the boundary}
	
In this section, we introduce some probabilistic objects that will appear in the proof, together with useful estimates.
In particular, we prove Propositions~\ref{prop:unifbound} and ~\ref{prop:unifcone}.
In the following $B(x,R)$ denotes the discrete Euclidean ball centered at~$x$ of radius~$R$.

\subsection{Rough bounds on the first eigenvalue}

In the proof of Theorem~\ref{thm:reg-cone}, we only need very rough (and easy) bounds on the first eigenvalue~$\lambda_N$, that we collect in the following lemma.

\begin{lemma}
\label{lem:eigenvalue}
Let $\Omega$ be an open and bounded set and $\Omega_N$ the connected component of~$0$ in $(N\Omega) \cap \bbZ^d$. 
Then there are two constants $c_\Omega,c_\Omega'$ that depends only on $\Omega$ such that the principal eigenvalue $\lambda_N$ of the transition matrix $P_N$ of the simple random walk killed upon exiting $\Omega_N$ verifies
\[
1- \frac{c_\Omega}{N^2}\leq \lambda_N \leq  1- \frac{c_\Omega'}{N^2}\,.
\]
Equivalently, there are constants $0<\gamma_\Omega<\gamma_\Omega'<1$ such that $\gamma_\Omega\leq (\lambda_N)^{N^2} \leq \gamma'_\Omega$ uniformly in~$N$.
\end{lemma}

\begin{proof}
The proof is very simple: we simply use that $\Omega$ contains a ball $B^{\rm int}$ and is contained in a ball $B^{\rm ext}$.
Since the principal eigenvalue is monotone in the domain, we obtain that $\lambda_N$ is sandwiched between the principal eigenvalues of $B^{\rm ext}_N$ and $B^{\rm int}_N$.

It simply remains to see that the principal eigenvalue $\lambda_{B_N}$ of a ball verifies $1- \frac{c}{N^2}\leq \lambda_{B_N}\leq 1- \frac{c'}{N^2}$, which is classical, see~\cite[Chap. 6.9]{lawlerRandomWalkModern2010} or Remark~\ref{rem:eigenball} below.
\end{proof}

\subsection{Random walk and confined random walk}
\label{sec:cond-rw}

We let $(X_n)_{n\geq 0}$ a simple nearest-neighbor random walk on $\mathbb{Z}^d$, whose transition probabilities are $p(x,y)\defeq \frac{1}{2d} \ind_{\{x\sim y\}}$, and let $\mathcal{F}_n = \sigma(X_0,X_1,\ldots, X_n)$ be the associated filtration.
We denote by $\bP_x$ the law of the random walk when starting from $x$.
For a set $\Lambda \subset \mathbb{Z}^d$, we denote
\[
H_{\Lambda} \defeq \inf\{n\geq 0, X_n \in \Lambda\} 
\]
the hitting time of $\Lambda$, with $\inf \varnothing = +\infty$.

Then, it is well-known that the first eigenvector $\phi_N$ is linked to the survival probability for the walk killed on $\partial \Omega_N$, see e.g.\ \cite[Prop.~6.9.1]{lawlerRandomWalkModern2010}: fixing $N$ large enough, for any $x \in \Omega_N$ we have
\begin{equation}
\label{eq:phiN-limite-proba-survie}
	N^{-2d} \sum_{z \in \Omega_N} \phi_N(z) \times \phi_N(x)  = \lim_{t \to +\infty} \lambda_N^{-t} \, \bP_x(H_{\partial \Omega_N} > t)   \, .
\end{equation}

Let us stress that~\eqref{eq:phiN-limite-proba-survie} shows in particular that, for $x,y\in \Omega_N$, $x\sim y$,
\begin{equation}
\label{eq:conditioning}
\lim_{t\to\infty} \bP_x\big( X_1 =y \mid H_{\partial \Omega_N} > t \big) = \frac{1}{2d} \frac{\lambda_N^{-1}\phi_N(y)}{\phi_N(x)} =: \tilde p_N(x,y) \,,
\end{equation}
and justifies the definition~\eqref{eq:def-p_N} of the \textit{confined} random walk,  which corresponds to the random walk conditioned to remain (forever) in $\Omega_N$.

\begin{remark}
\label{rem:eigenball}
Note that~\eqref{eq:phiN-limite-proba-survie} shows that, for any $x\in \Omega_N$,
$\lambda_N  = \lim_{t\to\infty} \bP_x(H_{\partial \Omega_N} > t)^{1/t}$, which also shows that $(\lambda_N)^{N^2} = \lim_{t\to\infty} \bP_0(H_{\partial \Omega_N} > t N^2)^{1/t}$.
For instance, if $\Omega_N$ is a ball of radius~$r N$ (say centered at $0$), one can  easily verify using Markov's property that, for any $k\in \mathbb{N}$,
\begin{multline*}
\Big(\inf_{x\in \Omega_N, |x|<\frac12 rN} \bP_x\big(H_{\partial \Omega_N} > N^2, |X_{N^2}|<\tfrac12 r N  \big) \Big)^k \\
\leq \bP_0(H_{\partial \Omega_N} > k N^2) \leq  \Big(\sup_{x\in \Omega_N} \bP_x\big( H_{\partial \Omega_N} > N^2 \big) \Big)^k \,.
\end{multline*}
By the invariance principle, we find that there are two constants $\gamma<\gamma'<1$ (that depend on $r$) such that $\gamma<(\lambda_N)^{N^2} \leq \gamma'$, showing as a consequence that $1- \frac{c}{N^2} \leq \lambda_N \leq 1- \frac{c'}{N^2}$.
\end{remark}

\subsection{Gambler's ruin estimates and a priori bounds on \texorpdfstring{$|\phi_N(x)|$}{}}
	\label{sec:boundunif}

Let us state here a random walk estimate, based on classical gambler's ruin arguments; its proof is postponed to Section~\ref{sec:ruine-joueur}.
We then show how one can deduce Proposition~\ref{prop:unifbound} from it.

\begin{lemma}\label{lem:sortie-ruine-joueur}
Under the uniform exterior ball Assumption~\ref{hyp:D}, there is a constant $c > 0$ such that for all $N$ large enough, all $x \in \Omega_N$,
		\begin{equation*}
			\Pbf_x(H_{\partial \Omega_N} > N^2) \leq  C \, \frac{d(x,\partial \Omega_N)}{N} \, .
		\end{equation*}
\end{lemma}

\begin{proof}[Proof of Proposition~\ref{prop:unifbound}]
Let $t > N^2$. Using the Markov property, we can write
	\begin{equation*}
		\lambda_N^{-t} \Pbf_x(H_{\partial \Omega_N} > t) 
		\leq \lambda_N^{-N^2} \Pbf_x(H_{\partial \Omega_N} > N^2) 
		\times \lambda_N^{-(t-N^2)} \sup_{z \in \Omega_N} \Pbf_z(H_{\partial \Omega_N} > t-N^2) \, .
	\end{equation*}	
Notice that using Lemma~\ref{lem:eigenvalue}, the term  $(\lambda_N)^{-N^2}$ is bounded by some universal constant, independent of~$N$.
Applying Lemma~\ref{lem:sortie-ruine-joueur}, we therefore get that
\[
\lambda_N^{-t} \Pbf_x(H_{\partial \Omega_N} > t) \leq  C \, \frac{d(x,\partial \Omega_N)}{N} \times \sup_{z \in \Omega_N} \lambda_N^{N^2-t}  \Pbf^N_z(H_{\partial \Omega_N} > t-N^2) \, .
\]
We can now take the limit as $t \to +\infty$ on both sides and exchange the supremum on $z \in \Omega_N$ and the limit as $t \to +\infty$ (since $N$ is fixed). Applying \eqref{eq:phiN-limite-proba-survie} then yields
\begin{equation*}
	\phi_N(x) \leq  C \, \frac{d(x,\partial \Omega_N)}{N} \sup_{z \in \Omega_N} \phi_N(z) \,.
\end{equation*}
Now, it remains to show that there is a constant $C'$ such that
\begin{equation}\label{eq:phiN-bornee}
\sup_{z \in \Omega_N} \phi_N(z) \leq C' \,.
\end{equation}
The proof is almost identical to the one of Lemma A.1 in \cite{dingDistributionRandomWalk2021a}: observe that $\phi_N(X_t) \lambda_N^{-t}$ is a martingale and use the optimal stopping theorem at time $\tau_N \defeq N^2 \wedge H_{\partial \Omega_N}$ to get
\begin{equation*}
	\begin{split}
		\phi_N(z) &= \Ebf_z \big[ \phi_N(X_{\tau_N}) \lambda_N^{-\tau_N} \big] \\
		&= \lambda_N^{-N^2} \sum_{w \in \Omega_N} \phi_N(w) \Pbf_z(H_{\partial \Omega_N} > N^2, X_{N^2} = w) + \Ebf_z \Big[ 0 \cdot \lambda_N^{-H_{\partial \Omega_N}} \ind_{\{H_{\partial \Omega_N} \leq N^2\}} \Big] \, ,
	\end{split}
\end{equation*}
where the second term is zero since $\phi_N \equiv 0$ on $\partial \Omega_N$. Removing the constraint  $H_{\partial \Omega_N} > N^2$ and using the local limit theorem \cite[Theorem 2.1.3]{lawlerRandomWalkModern2010}, we get
\[	
\sup_{z \in \Omega_N} \phi_N(z) \leq \lambda_N^{-N^2} \sum_{w \in \Omega_N} \phi_N(w) \Pbf_z( X_{N^2} = w) \leq c \gamma_\Omega^{-1} \sum_{w \in \Omega_N} \phi_N(w) N^{-d}  \,.
	\]
Then, the Cauchy--Schwarz inequality yields
\[
	\sup_{z \in \Omega_N} \phi_N(z)\leq c N^{-d} \Big( |\Omega_N| \sum_{w \in \Omega_N} \phi_N^2(w) \Big)^{1/2} = c |\Omega_N|^{1/2} N^{-d/2} \, ,
\]
where we used the normalization given by \eqref{eq:def-phiN-L2} for the last equality. Since $\Omega$ is bounded, by definition of $\Omega_N$ we have that $|\Omega_N|^{1/2} N^{-d/2}$ is also bounded uniformly in $N$, hence proving \eqref{eq:phiN-bornee} and concluding the proof.
\end{proof}

Assuming only the uniform cone condition Assumption~\ref{hyp:cone}, the main theorem of~\cite{DW15} gives the following result, which replaces Lemma~\ref{lem:sortie-ruine-joueur} and somehow shows that it is (strictly) easier for a random walk to avoid a cone than a ball.

\begin{theorem}[\cite{DW15}]
	Under Assumption~\ref{hyp:cone}, there is a constant $C>0$ such that, for any $x\in \Omega_N$,
	\[
	\Pbf_x(H_{\partial \Omega_N} > N^2) \leq  C \, \Big( \frac{d(x,\partial \Omega_N)}{N} \Big)^{p}  \,,
	\] 
	where $p = p(\alpha)\in (0,1]$ is defined in~\eqref{def:p}.
\end{theorem}

\noindent
Using this inequality in the above proof, one obtains Proposition~\ref{prop:unifcone} instead of Proposition~\ref{prop:unifbound}.

\begin{remark}
\label{rem:conti1}
Let us stress that the proof of Propositions~\ref{prop:unifbound} and~\ref{prop:unifcone} hold in the continuum, \textit{i.e.}\ considering a Brownian motion in $\Omega$ instead of the simple random walk in $\Omega_N$.
Using standard gambler's ruin estimates for Brownian motion (or~\cite{banuelos1997brownian,deblassie1987} under the exterior cone condition), we get the following bounds on the first continuous eigenfunction: for every~$x\in \Omega$,
\begin{equation}
	\label{eq:unifboundconti}
		|\varphi_1(x)| \leq C d(x,\partial \Omega)^p\,,
\end{equation}
with $p=1$ under Assumption~\ref{hyp:D} and $p\in (0,1]$ from~\eqref{def:p}-\eqref{def:p2} under Assumption~\ref{hyp:cone}.
This proves in particular Theorem~\ref{thm:reg-cone}.

For the sake of completeness, we give in Section \ref{sec:ruine-joueur} the proof of a Brownian analogue to Lemma \ref{lem:sortie-ruine-joueur}. The continuum version of Proposition \ref{prop:unifbound} is then proven with the same proof as in the discrete setting.
\end{remark}

\section{Difference estimates for \texorpdfstring{$\phi_N$}{} and \texorpdfstring{$\varphi_1$}{} via couplings}
\label{sec:regularity}

In this section, we prove all of our regularity estimates: we first prove Theorem~\ref{thm:reg-cone}; we start with the simple random walk case.
Then we explain how the proof works for the Brownian motion and how it adapts to higher-order differences, proving Theorem~\ref{th:reg-cont-cone}.
Finally, we conclude by the proof of the higher-order differences in the discrete case, \textit{i.e.}\ Theorem~\ref{thm:higherorder-cone}.

\subsection{Single difference estimates and simple mirror coupling}
\label{sec:srw-couplings}

In this section, we prove Theorem~\ref{thm:reg-cone}, whose proof we divide into several steps.
First of all, we treat the case where $x,y$ are at distance $2$ from each other, which allows us to treat the case with an even distance between $x,y$; we then adapt it to the general case.

\paragraph{When $x,y$ are at distance $2$ from each other.}
We work with fixed $x,y\in \Omega_N$ at distance $2$ from each other and such that $d(x,\partial \Omega_N)$ is large enough; in the case where $d(x,\partial \Omega_N) \leq C$, then one simply uses Proposition~\ref{prop:unifbound} to get that $|\phi_N(x)-\phi_N(y)| \leq \max(\phi_N(x),\phi_N(y)) \leq  C N^{-1}$. 

\medskip
\noindent
{\it Step 1. Rewriting of $|\phi_N(x)-\phi_N(y)|$.}
Our starting point is to use the Feynman--Kac relation~\eqref{eq:PtildeP} to rewrite $|\phi_N(x)-\phi_N(y)|$.
Let us consider the discrete ball $B(z,R)$ centered at the point $z$ such that $x\sim z$ and $y\sim z$, and of radius \(R \leq \frac{d(z,\partial \Omega_N)}{2}\).
We also assume that \(R\leq \delta_\Omega N\) for some fixed (but small) constant $\delta_\Omega$ that only depends on the domain $\Omega$ (this simply might change the constant in Theorem~\ref{th:reg-cont-cone}).

We also denote $H_R\defeq H_{\partial B(z,R)}$ for simplicity.
Then, using the relation~\eqref{eq:PtildeP} with the stopping time \(H_R\) so that \(H_R < H_{\partial \Omega_N}\), we obtain the following Feynman--Kac formula (recall~\eqref{eq:Feynman}): for any $x \in B(z,R)$
\begin{equation}
	\label{eq:sum=1}
	\phi_N(x) = \Ebf_{x}\Big[ (\lambda_N)^{-H_R}\, \phi_N\big(X_{H_R} \big) \Big] \, .
\end{equation}
In particular, for $x,y$ with $x\sim z$, $y\sim z$, we have
\begin{equation}
\label{eq:diff-phi-diff-esp}
	\big| \phi_N(x) - \phi_N(y) \big| =  \bigg|\Ebf_{x}\Big[ (\lambda_N)^{-H_R}\, \phi_N\big(X_{H_R} \big) \Big] - \Ebf_{y}\Big[ (\lambda_N)^{-H_R} \, \phi_N\big(X_{H_R} \big) \Big] \bigg| \, .
\end{equation}

Our goal is now to estimate the difference of expectations in~\eqref{eq:diff-phi-diff-esp} thanks to a coupling argument, which works when $x,y$ are at distance $2$ (due to the periodicity of the random walk).

\medskip
\noindent
{\it Step 2. Coupling argument.}
We now construct a coupling of two random walks $X^{(1)},X^{(2)}$ respectively starting at $x$ and $y$.
The coupling that we use is the so-called mirror coupling. 
The idea is to consider the hyperplane $\mathcal{H}=\mathcal{H}_{x,y}$ which is the mediator between~$x$ and~$y$ (and goes through~$z$). 
We then take a random walk $X^{(1)}$ that starts from~$x$ and we let \(\tau = \tau_{\cH} = \inf\{ n \geq 0, X_n^{(1)} \in \cH \}\) be the hitting time of \(\cH\).
We then define~\(X^{(2)}\) the \textit{mirror} version of \(X^{(1)}\) with respect to \(\cH\) until time \(\tau\), after which we set \(X^{(2)}=X^{(1)}\).
More precisely, decomposing \(X_n^{(1)}\defeq X_{n}^{(1),\perp}+X_{n}^{(1),\parallel}\) where \(X_{n}^{(1),\perp} = \langle X_n^{(1)}, e_{\cH} \rangle  e_{\cH} \) with \(e_{\cH} = \frac12 (x-y)\) is the component perpendicular to~\(\cH\) and \(X_{n}^{(1),\parallel} = X_n^{(1)}- X_{n}^{(1),\perp}\) is the component parallel to \(\cH\), we define
\begin{equation}
	\label{def:mirror}
	X_n^{(2)} \defeq 
	\begin{cases}
	X_{n}^{(1),\perp} -X_{n}^{(1),\parallel} & \text{ if } n\leq \tau \,,\\
	X_n^{(1)} & \text{ if } n > \tau \,.
	\end{cases}
\end{equation}
Notice that \(\tau\) is also the hitting time of $\mathcal{H}$ for $X^{(2)}$ and also that $\tau = \min\{n, X_n^{(1)}=X_n^{(2)}\}$ is the meeting time of $X^{(1)},X^{(2)}$, see Figure~\ref{fig:Balls} for an illustration.
We denote by $\bP_{x,y}$ the joint law of $(X_n^{(1)},X_n^{(2)})_{n\geq 0}$ that we have just constructed, which is our mirror coupling.

\smallskip
We denote by $H_R^{1},H_R^{2}$ the hitting times of $\partial B(z,R)$ by $X^{(1)},X^{(2)}$ respectively, and we stress that, on the event $\{\tau \leq \min(H_R^{1},H_R^{2})\}$ (the coupling is \textit{successful}), we have that both $X^{(1)},X^{(2)}$ reach $\partial B(z,R)$ at the same time and at the same point; we refer to Figure~\ref{fig:Balls} for an illustration. 
In fact, under the mirror coupling, we also have $H_R^1=H_R^2$, but we keep the index of the walk in the notation for clarity.
Therefore, we obtain
\begin{multline*}
\bigg|\Ebf_{x}\Big[ (\lambda_N)^{-H_R} \, \phi_N\big(X_{H_R} \big) \Big] - \Ebf_{y}\Big[ (\lambda_N)^{-H_R}\, \phi_N\big(X_{H_R} \big) \Big] \bigg|\\
 \leq \Ebf_{x,y}\bigg[ \Big| (\lambda_N)^{-H_R^{1}} \,\phi_N\big(X_{H_R^{1}}^{(1)} \big)  - (\lambda_N)^{-H_R^{2}} \,\phi_N\big(X_{H_R^{2}}^{(2)} \big) \Big| \; \ind_{\{\tau > \min(H_R^{1},H_R^{2})\}}\bigg] \\
 \leq  \Ebf_{x}\Big[ (\lambda_N)^{-H_R}\, \phi_N\big(X_{H_R} \big) \ind_{\{\tau > H_R\}}\Big] + \Ebf_{y}\Big[  (\lambda_N)^{-H_R} \, \phi_N\big(X_{H_R} \big) \ind_{\{\tau > H_R\}}\Big]\, .
\end{multline*}

\begin{figure}
\begin{center}
\includegraphics[scale=0.77]{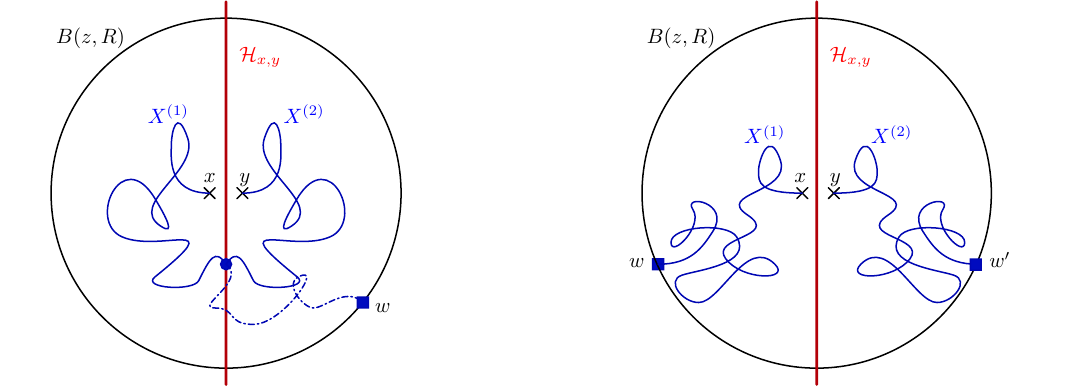}
\caption{On the left a successful coupling; the two walks exit $B(z,R)$ through the same point. On the right, the coupling fails, and the two walks reach $\partial B(z,R)$ before $\mathcal{H}$ and the coupling fails; the exit points are different (and symmetric).}
\label{fig:Balls}
\end{center}
\end{figure}

Now, we simply bound \(\phi_N(X_{H_R}) \leq \sup_{w\in B(z,R)} \phi_N(w)\), so that we get that there is a constant $C'>0$ such that, for $x,y$ at distance $2$,
\begin{equation}
	\label{eq:diff-final}
	|\phi_N(y)-\phi_N(x)| \leq 2  \sup_{w\in B(z,R)} \phi_N(w) \times \bE_x\Big[ (\lambda_N)^{-H_R}  \ind_{\{\tau >H_R\}} \Big] \,,
\end{equation}
using also the symmetry between $x$ and $y$.

\medskip
\noindent
{\it Step 3. Technical estimate and conclusion of the proof.}
It then only remains to control the expectation appearing in~\eqref{eq:diff-final}, which are (standard) simple random walk estimates (by translation and rotation invariance, we can take $\mathcal{H} =\{z=(z_1,\ldots, z_d)\in \mathbb{Z}^d, z_1=0\}$). 
We postpone the proof of the following lemma to Section~\ref{sec:lemhyperplan}, which collects some other useful random walk estimates.

\begin{lemma}[Gambler's ruin]
\label{lem:eviter-hyperplan-ruine}
There is a constant $c_d > 0$ (that depends only on the dimension) and a constant $C >0$ such that, for any $R\geq 1$ sufficiently large and any $x \in B(0,R/2)$, we have
\begin{equation}
\Ebf_{x} \Big[ e^{ c_d H_{R}/R^2 } \ind_{\{\tau_{\mathcal{H}} > H_{R}  \}} \Big] \leq C\, \frac{d(x,\mathcal{H})}{R} \, ,
\end{equation}
where $H_{R}$ is the hitting time of the ball $B(0,R)$ and $\tau_{\mathcal{H}}$ is the hitting time of the hyperplane $\mathcal{H} =\{z=(z_1,\ldots, z_d)\in \mathbb{Z}^d, z_1=0\}$.\end{lemma}
\noindent
Let us stress that the slight difficulty in this lemma comes from the term $e^{c_d H_R/R^2}$, which is unbounded: if this term were absent, it would be a standard gambler's ruin estimate.

\smallskip
We can then use that $\lambda_N^{-1} \leq e^{c_\Omega/N^2}$ for some constant $c_\Omega$, see Lemma~\ref{lem:eigenvalue}, so that we can bound $(\lambda_N)^{-H_{R}}\leq e^{c_\Omega H_R^x/N^2} \leq e^{c_\Omega \delta_\Omega^2 H_R/R^2}$ in~\eqref{eq:diff-final}, recalling also that we considered $R\leq \delta_\Omega N$.
Therefore, having fixed $\delta_\Omega$ small enough (how small depends on the domain $\Omega$), we can apply Lemma~\ref{lem:eviter-hyperplan-ruine} (with $x=e_1$ so $d(x,\mathcal{H})=1$) to obtain that for $x,y$ at distance $2$ from each other,
\begin{equation}
	\label{eq:lastdisplay1}
	|\phi_N(y) - \phi_N(x)| \leq \frac{2 C}{R}\,\sup_{w\in B(z,R)} \phi_N(w)  \,.
\end{equation}
This concludes the proof when $x,y$ are at distance $2$ from each other.

\paragraph{When $x,y$ are at an even distance from each other.}

We can now easily extend~\eqref{eq:lastdisplay1} to the case where $x,y$ are at an even distance from each other.
There exists a constant $C$ such that, for $z\in \Omega_N$ with $d(z,\Omega_N)$ large enough, letting $R=\frac{1}{2} d(z,\partial \Omega_N) \wedge (\delta_{\Omega} N)$ as above, we have that for any $x,y \in B(z,\frac12 R)$ at an even distance from each other
\begin{equation}
\label{lastdisplay}
 \big| \phi_N(y)-\phi_N(x) \big| \leq \frac{C}{R} \,  d(x,y) \, \sup_{w\in \partial B(z,R)} \phi_N(w)   \,.
\end{equation}
Indeed, one simply uses~\eqref{eq:lastdisplay1} together with the triangular inequality, using also that we have $d(w,\partial \Omega_N) \geq \frac14 d(z,\partial \Omega_N)$ for any $w\in B(z,\frac12 R)$ in a geodesic path from $x$ to $y$. 

\paragraph{When $x,y$ are at an odd distance from each other.}

We now turn to the ``odd'' case, which cannot directly be dealt with the mirror coupling.
The idea is to work with the so-called \emph{lazy} random walk: fix a laziness parameter $q \in (0,1/2]$, and consider the $q$-lazy random walk instead of a simple random walk --- we use the laziness of the walk to fall back on the case where starting points are at some even distance from each other.
Note that the $q$-lazy random walk killed upon exiting $\Omega_N$ has transition transition matrix $Q_N = q I + (1-q) P_N$; in other words, $Q_N(x,x)=q$ and $Q_N(x,y) = \frac{1-q}{2d}$ for $x\sim y$ with $x,y \in \Omega_N$.
One can readily see that $Q_N \phi_N = \lambda_N^{(q)}\phi_N$ with $\lambda_N^{(q)} = q+(1-q)\lambda_N$, so that $\phi_N$ is again the principal eigenvector of $Q_N$, with associated eigenvalue that verifies $1-\lambda_N^{(q)} = (1-q)(1-\lambda_N)$.

Let $z \in \Omega_N$ be such that $d(z,\Omega_N)$ is large enough, and let $R=\frac{1}{2} d(z,\partial \Omega_N) \wedge (\delta_{\Omega} N)$ as above, and let $x,y\in B(z,\frac14 R)$ with $|x-y|_1 \in 2\ZZ+1$.
Then, as in~\eqref{eq:diff-phi-diff-esp}, we obtain that
\begin{equation}
	\label{eq:diff-phi-diff-esp2}
		\big| \phi_N(x) - \phi_N(y) \big| \leq  \bigg|\Ebf_{x}^{(q)}\Big[ (\lambda_N^{(q)})^{-H_R}\, \phi_N\big(X_{H_R} \big) \Big] - \Ebf_{y}^{(q)}\Big[ (\lambda_N^{(q)})^{-H_R} \, \phi_N\big(X_{H_R} \big) \Big] \bigg| \, ,
\end{equation}
where here $(X_n)_{n\geq 0}$ is the $q$-lazy random walk, whose distribution is denoted~$\bP^{(q)}$.

We now introduce a coupling $\bP_{x,y}^{(q)}$ of two $q$-lazy walks $X^{(1)},X^{(2)}$ starting from $x,y$ respectively, in two steps: 
\begin{enumerate}[label=(\roman*)]
	\item First, we let $(\xi_i^{(1)})_{i\geq 1}$, $(\xi_{i}^{(2)})_{i\geq 1}$ be independent i.i.d.\ sequences of Bernoulli random variables of parameter $1-q$ and we set $T\defeq\min\{i\geq 1, \xi_i^{(1)}\neq\xi_i^{(2)}\}$. 
	Up until time $T$, we set $X_n^{(1)} = x +\sum_{i=1}^{n} \xi_i^{(1)} U_i$ and $X_n^{(2)}=y+\sum_{i=1}^{n} \xi_i^{(2)} U_i$, using the same steps $U_i$ (uniform in $\{\pm e_i, 1\leq i \leq d\}$) for the two walks; in other words, the two walks evolve in parallel until one stays still ($\xi=0$) but not the other ($\xi=1$).
	
	\item At time $T$, $X_T^{(1)}$ and $X_T^{(2)}$ are at an even distance one from the other: we then perform a mirror coupling with respect to the mediator hyperplane $\mathcal{H} = \mathcal{H}_{X_T^{(1)},X_T^{(2)}}$ between $X_T^{(1)},X_T^{(2)}$.
	Note that we might have $X_T^{(1)}=X_T^{(2)}$, in which case the two walks are coupled starting from time $T$.
\end{enumerate}

\noindent
Now, we can decompose~\eqref{eq:diff-phi-diff-esp2} according to whether $T\leq R/8$ or not: we have that 
\begin{multline}
	\label{withT}
	\big| \phi_N(x) - \phi_N(y) \big| \leq  \bigg| \Ebf_{x,y}^{(q)}\bigg[ \Big((\lambda_N^{(q)})^{-H_R^1}\, \phi_N\big(X_{H_R^1}^{(1)} \big)  - (\lambda_N^{(q)})^{-H_R^2}\phi_N\big(X_{H_R^2}^{(2)} \big) \Big) \ind_{\{T \leq  R/8\}}\bigg]  \bigg| \\
	+  \Ebf_{x,y}^{(q)}\bigg[ \Big|(\lambda_N^{(q)})^{-H_R^1}\, \phi_N\big(X_{H_R^1}^{(1)} \big)  - (\lambda_N^{(q)})^{-H_R^2}\phi_N\big(X_{H_R^2}^{(2)} \big) \Big| \ind_{\{T > R/8\}} \bigg] \,.
\end{multline}

For the first term in~\eqref{withT}, notice that $T<\min\{H_R^1,H_R^2\}$ (in fact, $X_T^{(1)},X_T^{(2)} \in B(z,R/2)$), so that applying the strong Markov property at time $T$, we have that it is equal to
\begin{multline*}
	\bigg| \Ebf_{x,y}^{(q)}\bigg[(\lambda_N)^{-T} \ind_{\{T \leq  R/8\}} \Ebf_{X_T^{(1)},X_T^{(2)}}^{(q)} \Big[ (\lambda_N^{(q)})^{-H_R^1}\, \phi_N\big(X_{H_R^1}^{(1)} \big)  - (\lambda_N^{(q)})^{-H_R^2}\phi_N\big(X_{H_R^2}^{(2)} \big)\Big] \bigg]  \bigg|\\
	\leq \Ebf_{x,y}^{(q)}\Big[(\lambda_N)^{-R/8}  \big|\phi_N(X_T^{(1)}) - \phi_N(X_T^{(2)})\big|\Big] \leq \frac{4 C}{R} \,  d(x,y) \, \sup_{w\in B(z,R)} \phi_N(w) \,.
\end{multline*}
For the last inequality, we have used that $(\lambda_N^{(q)})^{-R/8} \leq (\lambda_N^{(q)})^{- c N^2}\leq C'$ on one hand and~\eqref{lastdisplay} on the other hand (together with the fact that $d(X_T^{(1)},X_T^{(2)}) \leq d(x,y)+1 \leq 2 d(x,y)$, by construction).

For the second term in~\eqref{withT}, we bound $\phi_N(X_{H_R^i}^{(i)})$ by \(\sup_{w\in B(z,R)} \phi_N(w)\), and we apply the Markov property at time $R/8$: we get that this second term is bounded by
\[
	\sup_{w\in \partial B(z,R)} \phi_N(w) \times \Ebf_{x,y}^{(q)}\bigg[   (\lambda_N)^{-R/8} \ind_{\{T > R/8\}}  \Big( \bE_{X_{R/8}^{(1)}}\big[(\lambda_N^{(q)})^{-H_R} \big] + \bE_{X_{R/8}^{(2)}}\big[(\lambda_N^{(q)})^{-H_R} \big] \Big) \bigg] \,.
\]
Note that, as above, $(\lambda_N)^{-R/8} \leq (\lambda_N)^{-c N^2}$ is bounded by a constant.
For the remaining terms, we can use the following lemma, which is classical (we provide a short proof in Section~\ref{sec:lemsortie}).
\begin{lemma}[Exit time]
\label{lem:sortie-boule-lambda}
There is a constant $c_d > 0$ (that depends only on the dimension) and a constant $C >0$ such that, for any $ R\geq 1$ sufficiently large, we have
\begin{equation*}
		\sup_{u \in B(0, R)} \Ebf_u \Big[ e^{c_d H_R/R^2} \Big] \leq C \, ,
\end{equation*}
where $H_{R}$ is the hitting time of the ball $B(0,R)$.
\end{lemma}

Using that $(\lambda_N)^{-H_R} \leq e^{c_{\Omega} \delta_{\Omega}^2 H_R/R^2}$, this lemma shows that $\sup_{u\in B(z,R)} \Ebf_{u}\big[ \lambda_N^{-H_{R}} \big]$ is bounded by a constant, provided that $\delta_{\Omega}$ has been fixed small enough.
All together, we get that the second term in~\eqref{withT} is bounded by a constant times 
\[
\sup_{w\in B(z,R)} \phi_N(w) \times  \Pbf_{x,y}^{(q)}\big( T>R/8 \big) \leq   \frac{C}{R} \sup_{w\in B(z,R)} \phi_N(w)  \,,
\]
using also that $T$ is a geometric random variable with parameter $\theta = 2 q(1-q) <1$, so that $\bP(T>R/8)=\theta^{R/8} \leq c R^{-1}$. 
Combining all together inside~\eqref{withT}, this gives that for any $x,y\in B(z,R/4)$,
\[
|\phi_N(x)-\phi_N(y)| \leq \frac{C d(x,y)}{R} \sup_{w\in B(z,R)} \phi_N(w) \,,
\]
as desired.
\qed

\subsection{Brownian motion and higher order derivatives, multi-mirror coupling}
\label{sec:BM-couplings}

In this section, we adapt the proof to the continuous setting, and we prove the higher order derivative estimates in the continuous setting of Theorem~\ref{th:reg-cont-cone}.
One simply has to work with the Brownian motion instead of the simple random walk, with actually fewer technical difficulties (in particular, there is no parity issue); in the following, $(X_s)_{s\geq 0}$ denotes a $d$-dimensional standard Brownian motion, and we use similar notation as for the simple random walk to simplify the exposition.

As a warm-up, let us first explain how to adapt the proof of Theorem~\ref{thm:reg-cone}. We then turn to higher-order differences, where the coupling gets a bit more involved.

\paragraph*{Single differences, simple mirror coupling}

We let $z \in \Omega$ and $R \leq \frac12 d(z,\partial \Omega)\wedge \delta$. 
We consider $H_R$ the hitting time of $\partial B(0,R)$ by the Brownian motion $(z+X_s)_{s\geq 0}$,  and applying~\eqref{eq:Feynman} similarly as in~\eqref{eq:sum=1}, we get the following: for any $x,y \in B(0,R)$,
\begin{equation}
	\label{eq:sum=1-conti}
	|\varphi_1(x)- \varphi_1(y)| =  \bigg| \bE_x \Big[ e^{\mu_1 H_R} \varphi_1(z+X_{H_R}) \Big]  - \bE_{y} \Big[ e^{\mu_1 H_R} \varphi_1(z+X_{H_R})  \Big]  \bigg| \,,
\end{equation}
analogously to~\eqref{eq:diff-phi-diff-esp}.
Here, there is no periodicity issue, and we can directly define the mirror coupling: we let $\mathcal{H}$ be the mediator hyperplane between $x$ and $y$ and we define $X^{(1)}$, $X^{(2)}$ two mirror Brownian motion with respect to $\mathcal{H}$ until time $\tau \defeq \inf\{t>0, X_t^{(1)}=X_t^{(2)}\}$ (which is the hitting time of $\mathcal{H}$ for $X^{(1)}$ and $X^{(2)}$); then $X_t^{(1)} = X_t^{(2)}$ for any $t\geq \tau$.
Under this coupling, whose law is denoted by $\bP_{x,y}$, we have that $\varphi_1(X_{H_R^1}^{(1)}) = \varphi_1(X_{H_R^2}^{(2)})$ on the event $\{\tau < H_{R}\}$.

Then, similarly to~\eqref{eq:diff-final}, using translation and rotation invariance, we obtain that
\[
	|\varphi_1(x)- \varphi_1(y)| \leq  \sup_{w\in B(z,R)} \varphi_1(w) \times  \Big( \bE_x \Big[ e^{\mu_1 H_R} \ind_{\{\tau >H_R\}} \Big] +\bE_y \Big[ e^{\mu_1 H_R} \ind_{\{\tau >H_R\}} \Big]  \Big)\,.
\]

Now, we use the following result, analogous to Lemmas~\ref{lem:eviter-hyperplan-ruine}-\ref{lem:sortie-boule-lambda}, whose proof is identical (actually simpler) as for the random walk.
For any $\mu>0$, there is a constant $C>0$ and some $\delta>0$ such that, for any $R \in (0,\delta]$ and any $x\in B(0,R)$, 
\begin{equation}
	\label{eq:lemmas-conti}
		\bE_x\Big[ e^{\mu H_R} \ind_{\{\tau_{\mathcal{H}} > H_R \}} \Big]  \leq C \frac{d(x,\mathcal{H})}{R} \,,\qquad \text{ and } \qquad
		\bE_x \Big[ e^{\mu H_R} \Big]  \leq C \,, 
\end{equation}
where $H_R$ is the hitting time of $\partial B(0,R)$ and $\tau_{\mathcal{H}}$ is the hitting time of~$\mathcal{H} = \{0\}\times \mathbb{R}^{d-1}$.

Using the strong Markov property at the hitting time $H_{R/2}^{w} = H_{\partial B(w,R/2)}$ with $w=\frac{x+y}{2}$, we obtain that
\[
\bE_x \Big[ e^{\mu_1 H_R} \ind_{\{\tau >H_R\}} \Big] \leq \bE_x \Big[ e^{\mu_1 H_{R/2}^{w}} \ind_{\{\tau> H_{R/2}^{w}\}}  \bE_{X_{H_{R/2}^w}} \big[e^{\mu_1 H_R} \big] \Big] 
\leq C^2 \, \frac{d(x,y)}{R} \,,
\]
where we have used~\eqref{eq:lemmas-conti} for the last inequality (using also translation and rotation invariance of the Brownian motion) to bound $\bE_{X_{H_{R/2}^w}} \big[e^{\mu_1 H_R} \big]$ by a constant (uniformly on~$X_{H_{R/2}^w}$), and the fact that $d(x,\mathcal{H}) = \frac12 d(x,y)$.
All together, this gives that
\[
	|\varphi_1(x)- \varphi_1(y)| \leq C \frac{d(x,y)}{r} \sup_{w\in B(z,R)} \varphi_1(w)  \,,
\]
which is the continuous analogous of Theorem~\ref{thm:reg-cone}, \textit{i.e.}\ Theorem~\ref{thm:reg-cone} in the case \(k=1\).

\paragraph*{Higher-order differences, multi-mirror coupling}

Let $z \in \Omega$ and $R\leq \frac12 d(z,\partial \Omega)\wedge \delta$, as above.
Our goal is to show that there exists a constant $C$ (independent of $z$) such that, for any $h$ small enough, for any $i_1,\ldots, i_k\in \{1,\ldots, d\}$, we have that 
\begin{equation}
	\label{goal-kderiv}
	D^{(h)}_{i_1,\ldots, i_k} \varphi_1 (z) \leq  \Big(\frac{C kh}{R}\Big)^{k} \sup_{w\in B(z,R)} \varphi_1(w) \,.
\end{equation}

\smallskip
\noindent
{\it Step 0. Preliminaries. }
Our starting point is the formula~\eqref{derivatives}, which gives that
\[
	(2h)^k D^{(h)}_{i_1,\ldots, i_k} \varphi_1 (z) = \sum_{\alpha \in \{+1,-1\}^k} \sign(\alpha)\, \varphi_1 \big( z + x_{\alpha}\big) \, ,\quad \text{ with } x_{\alpha} = x_{\alpha}^{(h)} = h \sum_{j=1}^k \alpha_j e_{i_j} \,.
\]
We will work with $h$ small enough so that $kh \leq R$, so in particular all points $z+x_{\alpha}$ belong to $B(z,R)$.
Similarly to~\eqref{eq:sum=1-conti}, we therefore can write that 
\begin{equation*}
	(2h)^k D^{(h)}_{i_1,\ldots, i_k} \varphi_1 (z) = \sum_{\alpha \in \{+1,-1\}^k} \sign(\alpha)\, \bE_{x_{\alpha}}\Big[ e^{\mu_1 H_R} \,  \varphi_1\big( z+X_{H_R} \big) \Big] \,,
\end{equation*}
where $H_R$ is the hitting time of $\partial B(0,R)$.
Our goal is now to construct a coupling $\hat \bP$ of~$2^k$ Brownian motions, indexed by $\alpha\in \{-1,+1\}^k$ (we denote them $X^{(\alpha)}$), with respective starting points~$x_{\alpha}$. 
With such a coupling, we can rewrite:
\[
(2h)^k D^{(h)}_{i_1,\ldots, i_k} \varphi_1 (z) =  \hat \bE \bigg[ \sum_{\alpha \in \{+1,-1\}^k} \sign(\alpha) \,  e^{\mu_1 H_R^{\alpha}}\, \varphi_1\big( z+X_{H_R}^{(\alpha)}\big) \bigg] \,,
\]
where we have denoted $H_R^{\alpha}$ the hitting time of $\partial B(0,R)$ by $X^{(\alpha)}$.

\smallskip\noindent
\textit{Step 1. Properties needed for the coupling. }
Let us now comment on the desired properties of the relevant coupling.
Under $\hat \bP$, we want to have a coupling time $\tau$ such that, at time $\tau$ there is a partition of $\{+1,-1\}^k$ into pairs $(\alpha_+,\alpha_-)$ which verify: 
\begin{itemize}[noitemsep]
	\item[(i)]~$X^{(\alpha_+)}_{\tau} = X^{(\alpha_-)}_{\tau}$; 
	\item[(ii)]~$\sign(\alpha_+)\sign(\alpha_-) =-1$.
\end{itemize}

\noindent
Coalescing the pairs $(X^{(\alpha_+)}, X^{(\alpha_-)})$ after time~$\tau$, \textit{i.e.}\ setting $X^{(\alpha_+)}_t=X^{(\alpha_-)}_t$ for $t>\tau$, we then get that on the event $\tau < \min\{ H_R^{\alpha_+},H_R^{\alpha_-} \}$ the two Brownian motions $X^{(\alpha_+)},X^{(\alpha_-)}$ exit the ball $B(0,R)$ at the same time $H_R^{\alpha_+}=H_R^{\alpha_-}$ and at the same point $X_{H_R}^{(\alpha_+)}=X_{H_R}^{(\alpha_-)}$.
Therefore, on the event $\tau < \min\{ H_R^{\alpha_+},H_R^{\alpha_-} \}$, we have that
\[
	\sign(\alpha_+) \, e^{\mu_1 H_R^{\alpha_+}}\, \varphi_1\big( z+X_{H_R}^{(\alpha_+)}\big) + \sign(\alpha_-) \, e^{\mu_1 H_R^{\alpha_-}} \,\varphi_1\big( z+X_{H_R}^{(\alpha_-)}\big)  =0 \,,
\]
the two terms cancelling out since they have an opposite sign ($\sign(\alpha_+)\sign(\alpha_-)=-1$).

All together, with such a coupling, we would get the bound
\begin{multline*}
	(2h)^k \big| D^{(h)}_{i_1,\ldots, i_k} \varphi_1 (z) \big| \leq \sum_{\alpha \in \{+1,-1\}^k} \hat \bE\Big[ e^{\mu_1 H_R^{\alpha}} \varphi_1\big( z+X_{H_R}^{(\alpha)} \big)  \ind_{\{\tau > \min_{\alpha\in \{\pm1\}^k}H_R^{\alpha} \}}\Big] \\
	\leq \sup_{w\in B(z,R)} \varphi_1(w)  \sum_{\alpha \in \{+1,-1\}^k} \hat \bE\Big[ e^{\mu_1 H_R^{\alpha}} \ind_{\{\tau > \min_{\alpha\in \{\pm1\}^k}H_R^{\alpha} \}}\Big] \,.
\end{multline*}
It then remains to show that we can find a coupling that verifies the following property:
\begin{equation}
	\label{k-coupling}
	\hat \bE\Big[ e^{\mu_1 H_R^{\alpha}} \, \ind_{\{\tau > \min_{\alpha\in \{\pm1\}^k}\{H_R^{\alpha}\} \}}\Big] \leq  \Big( \frac{C k h}{R} \Big)^{k} \,.
\end{equation}
Plugging this in the display above would yield \eqref{goal-kderiv} as desired.

\smallskip
\noindent
{\it Step 2. Construction of the coupling. }
We let $\hat \bP$ be the law of $k$ independent standard Brownian motions $\{ W^{(j)} ,1\leq j\leq k\}$ with respective starting points $W^{(j)}_0= k h\,  e_{i_j}$.
We then define a \emph{generalized mirror coupling} (or \emph{multi-mirror coupling}) as follows. 

For any $j$, let us set $W^{(j,+1)} = W^{(j)}$ and $W^{(j,-1)}$ the mirror image of $W^{(j)}$ with respect to the hyperplane $\mathcal{H}_j = \{ (x_1,\ldots x_d) \in \mathbb{R}^d, x_{i_j}=0\}$; note that $W^{(j,-1)}$ starts from $-kh e_{i_j}$.
Then, for $\alpha \in \{-1,+1\}^k$, we set 
\begin{equation}
	\label{def:k-coupling}
	X^{(\alpha)} \defeq \frac{1}{\sqrt{k}} \sum_{j=1}^k W^{(j,\alpha_j)} \,,
\end{equation}
and we notice that under $\hat \bP$ the $(X^{(\alpha)})_{\alpha \in \{\pm 1\}^k}$ are indeed standard Brownian motions, with starting points $X_0^{(\alpha)} = x_{\alpha} = h \sum_{j=1}^k \alpha_j e_{i_j}$.

Define $\tau_j = \inf\{t>0, W^{(j)}_t \in \mathcal{H}_j \} = \inf\{t>0, W^{(j,+1)}_t =W^{(j,-1)}_t\}$ the time at which the $j$-th mirror coupling succeed, and let $\tau = \min\{\tau_j, 1\leq j \leq k\}$.
Now, for any $1\leq j \leq k$, let us define 
\[
	\alpha_\pm^j = (\alpha_1,\ldots, \alpha_{j-1}, \pm 1, \alpha_{j+1}, \ldots, \alpha_k) \,,
\]
so $\alpha_{\pm}^j$ is simply obtained from $\alpha$ by setting $\alpha_j$ to $\pm 1$.
Then for any $j\in \{1,\ldots, k\}$, this provides a ($j$-dependent) partition of $\{-1,+1\}^k$ into pairs $\{\alpha_+^j,\alpha_-^j\}$.
Let us stress right away that we have $\sign(\alpha_+^j) \sign(\alpha_-^j)=-1$, since $\alpha_{\pm}^j$ only differ by one sign (that of~$\alpha_j$).
Then, in the case where $W^{(j,+1)}_t =W^{(j,-1)}_t$, we have that for any $\alpha \in \{-1,+1\}^k$
\[
	X^{(\alpha_+^j)}_t = \frac{1}{\sqrt{k}} \sum_{j'\neq j} \alpha_j' W_t^{j,\alpha_{j'}} + \frac{1}{\sqrt{k}} W^{(j,+1)}_t = \frac{1}{\sqrt{k}} \sum_{j'\neq j} \alpha_j' W_t^{j,\alpha_{j'}} + \frac{1}{\sqrt{k}} W^{(j,-1)}_t = X^{(\alpha_-^j)}_t \,.
\]
Therefore, at time $\tau$, one can find some $j \in \{1,\ldots, k\}$ (the index of the successful coupling) and an associated partition $\{\alpha_+^j,\alpha_-^j\}$ which satisfies our requirements (i)-(ii) for the coupling. 
Note in fact that the partition depends on the index of the mirror coupling which succeed.

\smallskip
\noindent
{\it Step 3. Proof of~\eqref{k-coupling}. }
First of all, let us introduce 
\[
	T_R^{j} = \inf\big\{t>0, |W_t^{(j)}| =R \big\}\quad \text{ and } T \defeq \min_{1\leq j \leq k} T_R^{j} \,.
\]
Notice that, under the coupling $\hat\bP$, we have $|X_t^{(\alpha)}| \leq \frac1k \sum_{j=1}^k |W_t^{(j)}|$ so $|X_t^{(\alpha)}|< R$ for all $t< T$.
Therefore, we have that $T\leq \min_{\alpha\in \{\pm1\}^k}H_R^{\alpha}$, and we can bound the left-hand-side of~\eqref{k-coupling} by
\[
\hat \bE\Big[ e^{\mu_1 H_R^{\alpha}}\, \ind_{\{\tau>T \}}\Big] 
\leq \hat \bE\Big[   e^{\mu_1 T}\, \ind_{\{\tau>T\}}  \bE_{X_T^{(\alpha)}} \big[ e^{\mu_1 H_R^{\alpha}} \big] \Big] \,, 
\]
where we have used the strong Markov property at time $T\leq H_R^{\alpha}$ for the second inequality.
Then, we can use the second inequality in~\eqref{eq:lemmas-conti} with $X^{(\alpha)}$, to get that $\bE_{X_T^{(\alpha)}} \big[ e^{\mu_1 H_R^{\alpha}} \big]$ is bounded by a constant.
For the remaining term, writing that $T\leq \sum_{i=1}^k T_R^{j}$, we get that  
\[
\hat \bE\big[ e^{\mu_1 T} \ind_{\{\tau>T\}} \big] 
\leq \hat \bE\Big[ \prod_{j=1}^k e^{\mu_1 T_R^{j}} \ind_{\{\tau_j>T_R^{j}\}} \Big] =  \prod_{j=1}^k \hat \bE\Big[ e^{\mu_1 T_R^{j}} \ind_{\{\tau_j>T_R^{j}\}} \Big]\,,
\]
using that the $W^{(j)}$ are independent Brownian motions under $\hat \bP$.
Since the starting point of~$W^{(j)}$ is $k h\, e_{i_j}$, we can apply~\eqref{eq:lemmas-conti} to bound each term in the product by $C \frac{kh}{R}$, which concludes the proof of~\eqref{k-coupling}.
\qed

\subsection{Higher order differences in the discrete setting: Theorem~\ref{thm:higherorder-cone}}
\label{sec:rw-higher}

Let us stress that one can repeat the same argument as in Section~\ref{sec:BM-couplings} in the discrete setting, for the simple random walk.
Let us explain how the coupling works and what the differences with Section~\ref{sec:srw-couplings} are.
In the following, let $z\in \Omega_N$ with $d(z,\partial \Omega_N) \geq 4k$ and denote $R\leq \frac12 d(z,\partial \Omega_N) \wedge \delta$ (with \(R\geq 2k\)).

Consider the following coupling $\hat \bP$ of $2^k$ random walks.
Let $(S^{(j)})_{1\leq j \leq k}$ be $k$ independent simple random walks with starting points $S_0^{(j)} = e_{i_j}$, and let $S^{(j,+)}=S^{(j)}$ and $S^{(j,-)}$ the mirror image of $S^{(j)}$ with respect to the hyperplane $\mathcal{H}_{j} =  \{ (x_1,\ldots x_d) \in \mathbb{Z}^d, x_{i_j}=0\}$.
Then, for $\alpha \in \{+1,-1\}^k$, we define
\[
Y^{(\alpha)} = \sum_{j=1}^k S^{(j,\alpha_j)} \,.
\]
(Note that, contrary to~\eqref{def:k-coupling}, we do not divide by $k$, in order to keep $\mathbb{Z}^d$-valued random walks.)
Now, the starting points of $Y^{(\alpha)}$ are indeed $x_{\alpha} = \sum_{j=1}^k \alpha_j e_{i_j}$, but $(Y_n^{(\alpha)})_{n\geq 0}$ are \textit{not} simple random walks: they still are random walks, but with steps distributed as $V = U_1+\cdots+U_k$ with $(U_i)_{1\leq i\leq k}$ simple random walk steps (\textit{i.e.}\ independent random variables uniform in $\{\pm e_i, 1\leq i \leq d\}$).
In other words, we can write $(Y_n^{(\alpha)})_{n\geq 0} \stackrel{(d)}{=} (X_{kn}^{(\alpha)})_{n\geq 0}$ where $(X_n^{(\alpha)})$ is a simple random walk.

This is however not problematic for our purposes: the random walk $(Y_n)_{n\geq 0} \defeq (X_{kn})_{n\geq 0}$ have transition matrix $P_N^k$, so in fact $\phi_N$ is still its principal eigenfunction, but with associated eigenfunction $\lambda_N^k$.
Similarly to~\eqref{eq:sum=1}, we can therefore write for any \(x\in B(0,R)\),
\[
\phi_N(x) = \bE_x\Big[ (\lambda_N)^{-k\tilde H_{R}} \phi_N\big(z+Y_{\tilde H_R}\big) \Big] \,,
\]
where $\tilde H_{R}$ is the exit time of $B(0,R)$ by $z+Y$.

Then, under the coupling $\hat \bP$, the $k$-th order difference can be rewritten as
\begin{equation}
	\label{eq:k-coupling-srw}
	D_{i_1,\ldots, i_k} \phi_N(z) = \hat \bE\bigg[ \frac{1}{2^k}\sum_{\alpha\in \{+1,-1\}^k} \sign(\alpha) \, (\lambda_N)^{-k \tilde H_{R}^{\alpha}}\, \phi_N\big(z + Y_{\tilde H_R}^{(\alpha)}\big) \bigg] \,,
\end{equation}
with the obvious notation that $\tilde H_R^{\alpha}$ is the exit time of $B(0,R)$ by $z+Y^{(\alpha)}$.
As in the continuous setting, letting $\tau_j \defeq \min\{n\geq 0, S^{(j)}_n  \in \mathcal{H}_j \}$ and $\tau = \min\{\tau_j, 1\leq j \leq k\}$, we notice that at the time $\tau$ we can partition $\{\pm 1\}^k$ into pairs $\{\alpha_+,\alpha_-\}$ with $Y^{(\alpha_+)}_{\tau}=Y^{(\alpha_-)}_{\tau}$ and $\sign(\alpha_+)\sign(\alpha_-)=-1$, similarly as in the continuous setting.
Hence, coalescing the pairs $(Y^{(\alpha_+)},Y^{(\alpha_-)})$ after time $\tau$, we get that, on the event $\tau \leq \min_{\alpha\in \{\pm 1\}} \tilde H_R^{\alpha} $, all the terms in the sum cancel out.

We therefore end up with 
\begin{multline*}
	\big| D_{i_1,\ldots, i_k} \phi_N(z) \big| \leq   \frac{1}{2^k}\sum_{\alpha\in \{+1,-1\}^k}  \hat \bE\bigg[(\lambda_N)^{-k \tilde H_{R}^{\alpha}}\, \phi_N(z + Y_{\tilde H_R}) \ind_{\{\tau> \min_{\alpha\in \{\pm\}^k}\{\tilde H_R^{\alpha}\}\}}\bigg] \\
	\leq \frac{C}{2^k}\,   \sup_{w\in B(z,R)} \phi_N(w) \sum_{\alpha\in \{+1,-1\}^k}  \hat \bE\Big[(\lambda_N)^{-k \tilde H_{R}^{\alpha}}  \ind_{\{\tau>T\}}\Big] \,.
\end{multline*}
We have also used that $\tilde H_R^{\alpha} \geq T \defeq \min_{1\leq j\leq 1}T^{j}_{R/k}$ where we have defined $T^{j}_{R/k} = \min\{n\geq 0, S_n^{(j)} \notin B(0, R/k)\}$.
Applying the strong Markov property at time $T \leq \min_{\alpha \in \{\pm1\}^k} \tilde H_{R}^{\alpha}$, we get that 
\[
\begin{split}
	\hat \bE\Big[(\lambda_N)^{-k \tilde H_{R}^{\alpha}}  \ind_{\{\tau>T\}}\Big] & 
	= \hat \bE\Big[ (\lambda_N)^{-k T} \ind_{\{\tau>T\}}  \bE_{Y_T}\big[ (\lambda_N)^{-k \tilde H_{R}^{\alpha}} \big]  \Big] \\
	& \leq C \prod_{j=1}^k \bE\Big[ (\lambda_N)^{-k T_{R/k}^j} \ind_{\{\tau_j>T_{R/k}^j\}} \Big]\,,	
\end{split}
\]
where we have used Lemma~\ref{lem:sortie-boule-lambda} to bound the internal expectation by a constant (it easily adapts to the $k$-step random walk) and then the fact that $T\leq \sum_{j=1}^k T_{R/k}^j$ with the independence of the walks $(S^{(j)})_{1\leq j \leq k}$.
Now, we are left with estimates on the simple random walk $S^{(j)}$: thanks to Lemma~\ref{lem:eviter-hyperplan-ruine} we obtain
\[
\bE\Big[ (\lambda_N)^{-k T_{R/k}^j}\, \ind_{\{\tau_j>T_{R/k}^j\}} \Big] \leq \frac{C}{R/k} \,,
\]
using that $\lambda_N^{- k} \leq e^{c k/N^2} \leq e^{c_d (k/R)^2}$ since $R^2/N^2\leq c_d$, provided that $\delta$ is chosen small enough; note that we also used that $R/k\geq 2$.

All together, this proves that $| D_{i_1,\ldots, i_k} \phi_N(z)| \leq   \big(\frac{Ck}{R}\big)^k \sup_{w\in B(z,R)} \phi_N(w)$, as desired.
This concludes the proof of Theorem~\ref{thm:higherorder-cone}.
\qed

\begin{remark}[About directional $k$-th order differences]
	\label{rem:directional}
In the above, we only dealt with \textit{symmetric} differences defined in~\eqref{def:Di}-\eqref{derivatives}.
As noticed in Remark~\ref{rem:directional1}, we could also consider \textit{directional} differences $D_{i^+},D_{i^-}$.
Then, one can obtain a formula analogous to~\eqref{derivatives} for the higher order differences $D_{i_1^{\gep_1},\ldots, i_k^{\gep_k}}$, namely
\begin{equation}
	\label{derivatives2}
	D_{i_1^{\gep_1},\ldots, i_k^{\gep_k}} \psi (x) = \sum_{\alpha \in \{0,1\}^k} \sign(\alpha,\gep) \psi \Big( x + \sum_{j=1}^k \gep_j \alpha_j e_{i_j} \Big) \,,
\end{equation}
where $\sign(\alpha,\gep) = (-1)^m$ with $m$ the number of $j\in \{1,\ldots,k\}$ such that $\alpha_j=0,\gep_j=1$ or $\alpha_j=1,\gep_j=-1$.
One could still apply an identity of the type~\eqref{eq:k-coupling-srw}, but the difference here lies in the fact that the starting points of the different random walks $Y^{(\alpha)}$ should now be $x_{\alpha} =\sum_{j=1}^k \gep_j \alpha_j e_{i_j}$, which are not at distance $2$ from each other --- hence the mirror coupling of the random walks $(S^{(j)})_{1\leq j \leq k}$ does not fully work.

In order to circumvent this, one needs to work with lazy random walks, as in the last paragraph of Section~\ref{sec:srw-couplings}. 
The idea is to construct a coupling of $q$-lazy random walks $(S^{(j,+)},S^{(j,-)})_{1\leq j\leq k}$ which start respectively from $e_{i_j}$ and $0$, in two steps: first, let the random walks evolve in parallel until they all verify $S^{(j,+)}-S^{(j,-)} = 2 e_{i_j}$ (this takes a geometric number of random walk steps); after this, use the generalized mirror coupling described above. 
We do not write the details of this coupling and of the proof since it follows from straightforward adaptation of the above and do not bring much insight, but let us state the result that one would obtain.
There is a constant $C>0$ such that, for any $k\geq 1$ and any $i_1,\ldots, i_k \in \{1,\ldots, d\}$ and $\gep_1,\ldots, \gep_k \in \{\pm1\}^k$, for any $x\in \Omega_N$ with $d(x,\Omega_N) \geq 4 k$, and any \(R\leq \frac12 d(x,\partial \Omega_N)\) with \(R\geq 2k\), we have
\begin{equation}
	\label{eq:directionaldiff}
	\big| D_{i_1^{\gep_1},\ldots, i_k^{\gep_k}} \phi_N (x) \big| \leq \Big(\frac{Ck}{R}\Big)^k \sup_{w\in B(z,R)} \phi_N(w)\,.
\end{equation}
This is the analogue of Theorem~\ref{thm:higherorder-cone} for directional differences.
\end{remark}

\section{Simple random walk (and Brownian motion) estimates}
	\label{sec:rw}
	
The main goal of this section is to prove Lemmas~\ref{lem:eviter-hyperplan-ruine} and \ref{lem:sortie-boule-lambda}, but we start with the proof of Lemma~\ref{lem:sortie-ruine-joueur}, which is a classical gambler's ruin estimate.
We focus on the estimates for simple random walks, since the estimates for the Brownian motion (see~\eqref{eq:lemmas-conti}) are identical  (in fact, proofs are simpler).
We use the notation $B_R\defeq B(0,R)$ and we will denote $\hat H_R \defeq \min\{n \geq 0, X_n \in B_R\}$ and $\check H_R \defeq \min\{n \geq 0, X_n \notin B_R\}$, or simply $H_R = \check H_R$ if the random walk starts inside $B_R$ (according to the previous notation).

	\subsection{Gambler's ruin and escaping from large balls}
	\label{sec:ruine-joueur}
	
Let us first give a technical lemma on gambler's ruin probabilities, that we mostly deduce from well-known results (our key reference is~\cite{lawler2013intersections}).

\begin{lemma}
\label{lem:arg-martingale+survie-anneau}
	Fix $\alpha > 1$. There are constants $c_1,c_2$ (depending on $\alpha-1$), such that for all $R$ large enough, for all $x \in B_{\alpha R}\setminus B_R$, we have
	\begin{equation}
	\label{eq:ruine-joueur}
		\Pbf_x \big( \hat H_{R}  > \check H_{\alpha R} \big) \leq c_1 \frac{d(x,B_R)}{R}\,  \,,
	\end{equation}
and also
	\begin{equation}
	\label{eq:survie-proche-anneau}
	 \Pbf_x \big( \hat H_{R} \wedge \check H_{\alpha R} \geq  R^2 \big) \leq c_2 \frac{d(x,B_R)}{R} \, .
	\end{equation}
\end{lemma}
	
\begin{remark}
\label{rem:anneau-inverse}
Lemma~\ref{lem:arg-martingale+survie-anneau} is useful when the point $x$ is closer to $\partial B_R$ than $\partial B_{\alpha R}$.
When this is not the case, we can apply the same lemma but with different balls, to obtain\footnote{One simply needs to replace $B_R$ by a ball $\tilde B_R$ tangent to $B_{\alpha R}$ such that $d(x,\tilde B_R) = d(x,\partial B_R)$, and $B_{\alpha R}$ by a ball $\tilde B_{(1+\alpha)R}$ with the same center as $\tilde B_R$ but with a large radius so that $B_R \subset \tilde B_{(1+\alpha)R}$.}
\begin{equation}
\label{eq:anneau-inverse}
\Pbf_x \big( \hat H_{R}  > \check H_{\alpha R} \big) \leq c_1' \frac{d(x,\partial B_{\alpha R})}{R} \,,\qquad
 \Pbf_x \big( \hat H_{R} \wedge \check H_{\alpha R} \geq  R^2 \big) \leq c_2 \frac{d(x,\partial B_{\alpha R})}{R} \,.
\end{equation}
\end{remark}

\begin{proof}[Proof of \eqref{eq:ruine-joueur}]
We use Proposition 1.5.10 in \cite{lawler2013intersections}, which gives the following estimate in dimension $d\geq 3$: let $x \in B_{\alpha R}\setminus B_R$, then
\begin{equation}\label{eq:ruine-ito-discret}
	\Pbf_x( \hat H_{R} < \check H_{\alpha R}) =  \frac{|x|^{2-d} - \big(\alpha R \big)^{2-d} + \grdO(R^{1-d})}{ R^{2-d} - \big(\alpha R\big)^{2-d}} \, .
\end{equation}
Injecting $|x| = R+\ell$ with $\ell = |x|-R$ (we may assume that $\frac{\ell}{R}\leq 1/2$ otherwise the bound is trivial), this yields 
\[
\Pbf_x(\hat H_{R} > \check H_{\alpha  R}) = 1- \frac{(1 + \ell R^{-1})^{2-d} - \alpha^{2-d} + \grdO(R^{-1})}{1 - \alpha^{2-d}} \leq c_{\alpha,d} \big( \ell R^{-1} + \grdO(R^{-1}) \big)\, ,
\]
which is the desired result.
		
In dimension $d=2$, we have from \cite[Prop.~6.4.1]{lawlerRandomWalkModern2010} that, analogously as above,
\[
	\Pbf_x(\hat H_{R} < \check H_{\alpha R}) =  \frac{ \ln( \alpha  R) - \ln |x| + \grdO(R^{-1})}{\ln( \alpha  R) - \ln R} \,.
\]
Setting again $|x| = R+\ell$, we get after simplifications that
\[
\Pbf_x(\hat H_{R} > \check H_{\alpha R}) = \frac{\ln(1+\ell R^{-1}) + \grdO(R^{-1})}{\ln \alpha} \leq c_{\alpha ,d} \big( \ell R^{-1} + \grdO(R^{-1}) \big)\, ,
\]
as needed.
\end{proof}
	
\begin{proof}[Proof of \eqref{eq:survie-proche-anneau}]
The proof relies on the usual martingale argument.
We fix some $x\in B_{\alpha R} \setminus B_R$ such that $|x|-R \leq \delta R$, with $\delta=\delta_{\alpha} <\frac12 \alpha$ small enough (but fixed) so that in~\eqref{eq:ruine-joueur} we have $\bP_x(\hat H_R < \check H_{\alpha R}) \leq \frac12$; note that the bound~\eqref{eq:survie-proche-anneau} is trivial in the case $|x|-R >(2c_1)^{-1} R$.

Let us write for simplicity $T\defeq \min\{ \hat H_{R} , \check H_{\alpha R}\}$, and consider the martingale $|X_{t \wedge T}|^2 - t \wedge T$.
Applying the stopping time theorem, we get that
\begin{equation}
	|x|^2 = \bE_x \big[ |X_{t \wedge T}|^2 - t \wedge T \big] \xrightarrow[]{t \uparrow \infty}\bE_x\big[ |X_{T}|^2 - T \big] \, ,
\end{equation}
where we have used dominated and monotonous convergence as we took the limit $t\uparrow \infty$.
Splitting the last expectation according to whether $X_{T} \in B_R$ or not, we have
	\[ |x|^2 = \bP_x\big(X_{T} \in B_{R}\big) \bE_x\big[ |X_{T}|^2 - T \, \big| \, X_{T} \in B_{R} \big] + \bP_x\big(X_{T} \notin B_{R}\big) \bE_x\big[ |X_{T}|^2 - T \, \big| \, X_{T} \notin B_{R} \big]  \, . 
	\]
Rearranging the terms, we obtain
\begin{multline*}
	\probaRW{x}{X_{T} \in B_{R}} \espRW{z}{T \, \big| \, X_{T} \in B_{R}} =  \probaRW{x}{X_{T} \in B_{T}} \espRW{x}{|X_{T}|^2 - |x|^2 \, \big| \, X_{T} \in B_{R}} \\
	+ \probaRW{x}{X_{T} \not\in B_{R}} \espRW{x}{|X_{R}|^2 - T  - |x|^2 \, \big| \, X_{T} \not\in B_{R}} \, .
\end{multline*}
	Since we took $x \notin B_{R}$, on the event $\{X_{T} \in B_{R}\}$ we have $|X_{T}|^2 \leq |x|^2$, so we can bound the first term by $0$.
Using also that $T \geq 0$, we end up with
\[
\probaRW{x}{X_{T} \in B_{T}} \espRW{x}{T \, \big| \, X_{T} \in B_{R}} \leq \probaRW{x}{X_{T} \notin B_{R}} \espRW{x}{|X_{T}|^2 - |x|^2 \, \big| \, X_{T} \notin B_{R}} \, . 
\]
Note that on the event $\{X_{T} \notin B_{R}\}$ and since we have fixed $x$ verifying $|x| -R \leq \delta R$, we have $|X_{T}|^2 - |x|^2 \leq c_{\alpha} R^2$, and recall that we chose $\delta$ small enough so that $\probaRW{x}{X_{T} \in B_{R}} \geq \frac12$ in~\eqref{eq:ruine-joueur}. Therefore, we obtain
\begin{equation}
\label{eq:boundET}
	\espRW{x}{T \, | \, X_{T} \in B_{R}} \leq 2 c_{\alpha} R^2 \probaRW{z}{X_{T} \not\in B_{R}} \leq 2 c_{\alpha} c_1 R (|x|-R) \, ,
\end{equation}
where we have used~\eqref{eq:ruine-joueur} for the last inequality.
We thus get that \(\bP_x(T>  R^2)\) is equal to
\[
\begin{split}
\bP_x(X_T\in B_R)&  \bP_x\big( T > R^2 \, | \, X_{T} \in B_{R}\big) + \bP_x(X_T\notin B_R) \bP_x\big( T >R^2 \, | \, X_{T} \notin B_{R}\big) \\
&\leq \bP_x\big( T >R^2 \, | \, X_{T} \in B_{R}\big) + \bP_x(X_T\notin B_R) \leq 4 c_{\alpha} c_1 \frac{ |x|-R}{R} + c_1 \frac{|x|-R}{R} \,,
\end{split}
\]
where in the last line we have used Markov's inequality together with~\eqref{eq:boundET} for the first term, and~\eqref{eq:ruine-joueur} for the second term.
This concludes the proof of~\eqref{eq:survie-proche-anneau}.
\end{proof}

We are now ready to conclude the proof of Lemma~\ref{lem:sortie-ruine-joueur}, thanks to Assumption~\ref{hyp:D}.

\begin{proof}[Proof of Lemma \ref{lem:sortie-ruine-joueur}]
Using Assumption~\ref{hyp:D}, there is some $\eps_0 > 0$ for which, for all $x \in \Omega_N$, there exists $z \notin \Omega_N$ such that the ball $B(z,\eps_0 N)$ is in $\Omega_N^c$ and $d(x,B(z,\eps_0 N)) \leq 2 d(x,\partial \Omega_N)$.
Now, let $A$ be large enough, so that $\Omega_N \subset B(z,A N)$.

Now, we simply observe that $H_{\partial \Omega_N} \leq \min\{H_{B(z,\eps_0 N)},H_{B(z,AN)^c}\}$, so that using Lemma~\ref{lem:arg-martingale+survie-anneau}-\eqref{eq:survie-proche-anneau} (with $\alpha \defeq A/\gep_0$), we get that 
\[
\Pbf_x(H_{\partial \Omega_N} > N^2) \leq c \frac{B(z,\eps_0 N)}{\gep_0 N} \leq c' \frac{d(x,\partial \Omega_N)}{N} \,.
\]
This concludes the proof of Lemma~\ref{lem:sortie-ruine-joueur}.
\end{proof}

\paragraph*{A Brownian version of Lemma \ref{lem:sortie-ruine-joueur}}

We also have the following continuous version of Lemma \ref{lem:sortie-ruine-joueur}.
Under the uniform exterior ball Assumption~\ref{hyp:D}, there is some $c > 0$ such that
\begin{equation}
	\label{eq:survie-MB}
	\Pbf_x \big( H_{\partial \Omega} > 1 \big) \leq c\, d(x,\partial \Omega) \, .
\end{equation}
Let us now explain how this follows from classical Brownian estimates.

We first begin by proving an analogue of~Lemma \ref{lem:arg-martingale+survie-anneau}. 
Using Itô formula, we have an exact formula for \(\Pbf_x(H_R<H_{\alpha R})\): indeed, \eqref{eq:ruine-ito-discret} holds without the $\grdO(R^{-1})$. 
Again we can assume that $\ell / R$ is small, from which we get \eqref{eq:ruine-joueur}. Now, since $B_t^2 - t$ is a martingale, the proof of \eqref{eq:survie-proche-anneau} also holds in the continuous.

Finally, let $x \in \Omega$ and $z \in \Omega^c$ such that $B(z,\eps_0) \subset \Omega^c$ and $d(x,B(z,\eps_0)) \leq 2 d(x,\partial \Omega)$. If $A > 0$ is large enough so that $\Omega \subset B(z,A)$, we have $H_\Omega \leq \min \big\{ H_{B(z,\eps_0)} , H_{B(z,A)^c} \big\}$. Therefore, $\Pbf_x ( H_{\partial \Omega} > 1 ) \leq \frac{c}{\eps_0 A} d(x,\partial \Omega)$, which is what we wanted to prove.

\subsection{Proof of Lemma \ref{lem:sortie-boule-lambda}}
\label{sec:lemsortie}

The proof is easy. First of all, notice that if we define $\gamma_R \defeq \sup_{v\in B_R} \bP_v(H_R >R^2)$, then we have that $\gamma_R \leq \gamma_d$ for some constant $\gamma_d <1$, uniformly in $R$ large enough --- this is due to the invariance principle, we have $\lim_{R\to\infty}\gamma_R = \bP_0(\sup_{s\in [0,1]}|X_s|<1)$ with $(X_s)_{s\geq 0}$ a Brownian motion.

Then, applying the Markov property iteratively, we have that $\bP_u(H_R \geq  j R^2) \leq (\gamma_d)^j$ for any integer $j$ (uniformly in $u\in B(0,R)$), so that
\[
\bE_u\Big[ e^{c_d H_R /R^2} \Big] \leq \sum_{j\geq 0} e^{c_d (j+1)} \bP_u\big( H_R/R^2 \in [j,j+1) \big) \leq e^{c_d} \sum_{j\geq 0} \big( e^{c_d} \gamma_d \big)^{j}\,.
\]
This is bounded by a constant, for instance if $c_d \defeq \frac12 \ln \gamma_d^{-1}$ so that $e^{c_d}\gamma_d = \gamma_d^{1/2}<1$.
\qed

\subsection{Proof of Lemma \ref{lem:eviter-hyperplan-ruine}}
\label{sec:lemhyperplan}

The proof will mostly rely on Lemma~\ref{lem:arg-martingale+survie-anneau}.
Decomposing over the value of the integer part of $R^{-2} H_R$, we get the bound
\[
\Ebf_0 \Big[ e^{c_d H_{R}/R^2} \ind_{\{  \tau_{\mathcal{H}} > H_R \}} \Big] \leq  e^{c_d} \Pbf_0\big( \tau_{\mathcal{H}}> H_R \big)+ \sum_{k \geq 1} e^{c_d(k+1)} \Pbf_0\big( \tau_{\mathcal{H}}> H_R \geq k R^2 \big) \, .
\]
For any $k\geq 1$, we can use the Markov property at time $R^2$ to get that 
\begin{multline*}
\Pbf_0\big( \tau_{\mathcal{H}}> H_R \geq k R^2 \big)  \leq \Pbf_0\big( \tau_{\mathcal{H}}\wedge H_R \geq R^2 \big) \sup_{v\in B_R} \bP_v\big( H_R \geq (k-1)R^2 \big) \\
\leq \Pbf_0\big( \tau_{\mathcal{H}}\wedge H_R \geq R^2 \big) \times (\gamma_d)^{k-1} \,,
\end{multline*}
with $\gamma_R \defeq \sup_{v\in B_R} \bP_v(H_R >R^2) \leq \gamma_d <1$ is as in the proof of Lemma~\ref{lem:sortie-boule-lambda}.
All together, we obtain that
\[
\Ebf_0 \Big[ e^{c_d H_{R}/R^2} \ind_{\{  \tau_{\mathcal{H}} > H_R \}} \Big] \leq  e^{c_d} \Pbf_0\big( \tau_{\mathcal{H}}> H_R \big)+ \Pbf_0\big( \tau_{\mathcal{H}}\wedge H_R \geq R^2 \big)  \sum_{k \geq 1} e^{c_d} (e^{c_d} \gamma_d)^k \,,
\]
and the last sum is bounded by a constant choosing for instance $c_d \defeq \frac12 \ln \gamma_d^{-1}$.
It therefore remains to estimate the two probabilities in the above display.

In order to be in position to apply Lemma~\ref{lem:arg-martingale+survie-anneau}, we introduce some new sets.
Let $z \in \mathbb{Z}^2$ be such that the ball $\tilde B_{R/2} \defeq B(z,\frac12 R)$ of radius $R/2$ is tangent to $\mathcal{H}$ on the other side of $0$ (so in particular $d(0,\tilde B_{R/2}) =1$), and let also $\tilde B_R \defeq B(z,R)$ and $\tilde B_{2R} \defeq B(z,2R)$.
Then, starting from $0$, by construction we have that $H_R \geq H_{\partial \tilde B_R}$ on the event $\tau_{\mathcal{H}} > H_R$, and also $\tau_{\mathcal{H}} \leq H_{\tilde B_{R/2}}$, so that we get
\[
\bP_0\big( \tau_{\mathcal{H}} > H_R \big)\leq \bP_0\big( \tau_{\mathcal{H}} > H_{\partial \tilde B_{R}} \big)  \leq \bP_0\big( H_{\tilde B_{R/2}} > H_{\partial \tilde B_{R}} \big) \leq \frac{c}{R} \,.
\]
For the last inequality, we have used Lemma~\ref{lem:arg-martingale+survie-anneau}-\eqref{eq:ruine-joueur} (with $\alpha=2$).

On the other hand, we also have by construction that $H_{\tilde B_{R/2}} \wedge H_{ \partial \tilde B_{2 R}} \geq \tau_{\mathcal{H}} \wedge H_R$, so by Lemma~\ref{lem:arg-martingale+survie-anneau}-\eqref{eq:survie-proche-anneau} (with $\alpha=4$), we obtain 
\[
\Pbf_0\big( \tau_{\mathcal{H}}\wedge H_R \geq R^2 \big) \leq \Pbf_0\big( H_{\tilde B_{R/2}} \wedge H_{ \partial \tilde B_{2 R}} \geq (R/2)^2 \big)  \leq \frac{c'}{R} \,.
\]
This concludes the proof of Lemma~\ref{lem:eviter-hyperplan-ruine}.
\qed

\begin{appendix}

\section{Outlining and adapting Bramble and Hubbard's proof of \texorpdfstring{$L^2$}{L2} and \texorpdfstring{$L^\infty$}{Linfty} convergence}
\label{app:Bramble}

In this appendix, we prove Theorems~\ref{th:L2convergence} and~\ref{th:sup-convergence}.
The proof of Bramble and Hubbard \cite{bramble1968effects} applies with very minor modifications, but we provide a complete summary for the sake of completeness.
In order to make the proof closer to that of~\cite{bramble1968effects}, we work with the notation of the discrete Dirichlet problem~\eqref{eq:approxDirichlet}: we consider the discrete Laplacian $\Delta^{(h)}$ on $\Omega^{(h)}$ with Dirichlet boundary conditions, and we denote by $\varphi_1^{(h)}$ its principal eigenfunction.
Our main goal is therefore to show
\begin{equation}
	\label{eq:L2-convergence}
	h^d \sum_{x\in \Omega^{(h)}} \big|\varphi_1^{(h)}(x)-\varphi_1(x) \big|^2 \leq \kappa h^{p} \,,
\end{equation}
in a first time, and
\begin{equation}
	\label{eq:sup-convergence}
	\lim_{h \to 0} \sup_{x \in \Omega^{(h)}} \big| \varphi^{(h)}_1(x) - \varphi_1(x) \big| = 0 \, .
\end{equation}
in a second time.

The proof of the $L^2$ convergence~\eqref{eq:L2-convergence} mostly consists in controlling the difference $|\mu_1^{(h)} - \mu_1|$ of the associated eigenvalues, from which the procedure described in \cite[Section 7]{bramble1968effects} will allows us to conclude.
It heavily relies on a key estimate on the derivative of $\varphi_1$ near the boundary (see \cite[Lem.~6.1-6.2]{bramble1968effects}) that Theorem~\ref{th:reg-cont-cone} (or more precisely Corollary~\ref{cor:derivatives}) will provide. 
The other key estimate is a bound on $\varphi_1^{(h)}$ near the boundary, which Proposition~\ref{prop:unifcone} provides. 
The $L^{\infty}$ convergence will follow from applying almost verbatim the proof of~\cite[Theorem~7.1]{bramble1968effects}.
We now provide a detailed outline of the proof.

\subsection{Eigenvalue convergence}

Before investigating the convergence of discrete eigenvalues, we first recall the minimum-maximum property: letting \(\mu_j^{(h)}\) the $j$-th discrete eigenvalue, we have
\begin{equation}\label{eq:minmax-property}
		\mu^{(h)}_j = \min_{f_1, \dots , f_j} \max_{\alpha_1, \dots , \alpha_j} \frac{1}{\| f \|_{2, h}^2} h^{d-2} \sum_{x \in h \ZZ^d} \sum_{i = 1}^d \big( f(x+h e_i) - f(x) \big)^2 \, ,
\end{equation}
where $\alpha_1 , \dots , \alpha_j$ are real numbers, $f_1, \dots , f_j$ are linearly independent mesh functions that vanish outside $\Omega^{(h)}$, and $f$ is defined as $f = \sum_{k = 1}^j \alpha_k f_k$. 
Note that the sum in \eqref{eq:minmax-property} in over the entire space and can be interpreted as $\| \nabla f \|_{2,h}^2$, thus a Dirichlet energy.

To get a bound on the difference of eigenvalues, we may use Weinberger's method~\cite{weinberger1959error} of comparing finite difference domains. 
Write $\Omega_\star^{(h)}$ for the interior of $\Omega^{(h)}$, so that $\Omega_\star^{(h)} \cup \partial \Omega_\star^{(h)} = \Omega^{(h)}$. 
We may consider the discrete Dirichlet problem
\begin{equation}\label{eq:discrete-Dirichlet-interior}
	\begin{cases}
		-\Delta^{(h)}  w^{(h)} = \mu_\star^{(h)}\, w^{(h)} & \quad \text{ on } \Omega_\star^{(h)} \,,\\ 
		\qquad \ \ w^{(h)}=0 & \quad\text{ on } \partial \Omega_\star^{(h)} \,,
	\end{cases}
\end{equation}
with $w : \Omega^{(h)} \rightarrow \RR$. We write $\mu_{\star,j}^{(h)}$ the ordered eigenvalues of the problem \eqref{eq:discrete-Dirichlet-interior}.

Using (3.5)-(3.6) in \cite{bramble1968effects}, which are derived from works of Weinberger and others, there are positive constants $c_1, c_2$ such that for all $j \in \{ 1 , \dots |\Omega^{(h)}| \}$,
\begin{equation}\label{eq:encadrement-vp-enveloppes}
	(1 - c_1 h^2)\mu^{(h)}_j \leq \mu_j \leq  (1 + c_2 h^2) \mu_{\star,j}^{(h)}\, .
\end{equation}
We now only need to control $|\mu^{(h)}_j - \mu_{\star,j}^{(h)}|$ in order to prove that $\mu^{(h)}_j \to \mu_j$. 
The proof is straightforward with the uniform exterior ball condition of Assumption~\ref{hyp:D}, while in the case of Assumption~\ref{hyp:cone}, the proof in \cite{bramble1968effects} remains valid.
Let us provide a few details.

\smallskip\noindent
\textit{A proof that $\mu^{(h)}_1 \to \mu_1$ under a uniform exterior ball condition.\ }	
Note that in \eqref{eq:minmax-property} if we take $j=1$ and $f = \varphi^{(h)}_1 \ind_{\Omega^{(h)}_\star}$ (this is (5.4) in \cite{weinberger1959error}) we get
\begin{equation*}
	\mu^{(h)}_1 \leq \mu_{\star,1}^{(h)} \leq \bigg( \mu^{(h)}_1 + h^{d-2} \sum_{x \in \partial \Omega_\star^{(h)}} \varphi^{(h)}_1(x)^2 \bigg) \, \Bigg/ \, \bigg( 1 - h^d \sum_{z \in \Omega^{(h)} \setminus \Omega_\star^{(h)}} \varphi^{(h)}_1(z)^2 \bigg) \, .
\end{equation*}
Observe that both $|\Omega^{(h)} \setminus \Omega_\star^{(h)}|$ and $|\partial \Omega_\star^{(h)}|$ are of order $h^{1-d}$. Therefore, using the bound on~$\varphi^{(h)}_1$ given by Proposition~\ref{prop:unifbound}, this yields
\begin{equation*}
	\mu^{(h)}_1 \leq \mu_{\star,1}^{(h)} \leq \frac{ \mu^{(h)}_1 + c h }{ 1 - c' h^3 } \leq \mu^{(h)}_1 + c'' h \, .
\end{equation*}
Recalling \eqref{eq:encadrement-vp-enveloppes} this gives $(1 - c_1 h^2)\mu^{(h)}_1 \leq \mu_1 \leq  (1 + c_2 h^2) (\mu^{(h)}_1 + c'' h)$, hence proving
\begin{equation}\label{eq:vitesse-cv-vp-pos-reach}
	|\mu_1 - \mu^{(h)}_1| \leq K h \,.
\end{equation}

\smallskip\noindent	
\textit{Generalizing to other Lipschitz domains or other eigenvalues.\ }
Under Assumption~\ref{hyp:cone}, applying the previous strategy yields $\mu^{(h)}_1 \leq \mu_{\star,1}^{(h)} \leq \mu^{(h)}_1 + c'' h^{2p-1}$ for \(p = p(\alpha)\) appearing in Proposition~\ref{prop:unifcone}. 
However, from the discussion that follows Proposition~\ref{prop:unifcone}, the exponent $p$ may be arbitrarily small, which means that $h^{2p-1}$ may diverge as $h \to 0$. 
Still, in the case of Lipschitz domains, one can prove (see \cite[Theorem 4.1]{bramble1968effects}) that $|\mu^{(h)}_j - \mu_{\star,j}^{(h)}| \to 0$ using a variation of the previous proof. 
This in fact only requires a Walsh approximation theorem for the domain $\Omega$, which roughly states that harmonic functions on $\Omega$ are close to harmonic functions on a slightly bigger domain; this is satisfied in the case of Lipschitz domains (Bramble and Hubbard refer to \cite[p. 281]{babuska1966numerical}).

\subsection{From eigenvalue to \texorpdfstring{$L^2$}{L2} convergence}
\label{sec:cv-L2}
	
After obtaining the convergence of eigenvalues, we can deduce a control on the $L^2$ error using the eigenvector decomposition.
A simple expansion and the $L^2$-normalisation of $\varphi^{(h)}_1$ implies
\begin{equation}\label{eq:L2-error-produit-scalaire}
	\| \varphi_1 - \varphi_1^{(h)} \|_{L^2,h}^2 = 1 + \| \varphi_1 \|_{L^2,h}^2 - 2 \langle \varphi_1, \varphi_1^{(h)} \rangle_h \, ,
\end{equation}
where $\langle \psi,\psi' \rangle_h = h^d \sum_{x \in \Omega^{(h)}} \psi(x) \psi'(x)$ is the scalar product in $L^2(\Omega^{(h)})$.
Controlling the first derivative of $\varphi_1$ thanks to Corollary~\ref{cor:derivative}, we easily get that $\| \varphi_1 \|_{L^2,h}^2 = 1 + O(h^p)$ as $h \to 0$. 
All that is left is thus to prove that $\langle \varphi_1, \varphi_1^{(h)} \rangle_h \to 1$ as $h \downarrow 0$, with a (polynomial) control on the decay rate.

We introduce the Green's function of the simple random walk on $h \ZZ^d$ killed on the boundary of $\Omega^{(h)}$, that is for $x,y \in \Omega^{(h)} \cup \partial \Omega^{(h)}$:
\begin{equation*}
	G^{(h)}(x,y) = \Ebf_x \Big[ \sum_{k = 0}^{\tau^{(h)}} \ind_{\{X_k = y\}} \Big] \quad \text{where} \quad \tau^{(h)} \defeq \inf \big\{ t \geq 0 \, : \, X_t \in \partial \Omega^{(h)} \big\} \, .
\end{equation*}
Note that for $x,y \in \Omega^{(h)}$ we have $G^{(h)}(x,y) = G^{(h)}(y,x)$. It is also well-known that $G^{(h)}$ is an inverse to the Laplace operator, in the sense that it satisfies
\begin{equation*}
	- \Delta^{(h)}_x G^{(h)}(x,y) \defeq \frac{h^{-2}}{2d} \sum_{|e|=1} \big[ G^{(h)}(x+he,y) - G^{(h)}(x,y) \big] = h^{-2} \delta_{x,y} \quad \text{ for } x \in \Omega^{(h)} \, ,
\end{equation*}
as well as $G^{(h)}(x,y) = \delta_{x,y}$ for $x \in \partial \Omega^{(h)}$.
From this, we deduce a crucial tool to study $G^{(h)}$ in the form of a discrete Poisson formula: let $f$ be a real-valued function on $\Omega^{(h)} \cup \partial \Omega^{(h)}$ and $x \in \Omega^{(h)} \cup \partial \Omega^{(h)}$, we have
\begin{equation}\label{eq:poisson-formula}
	f(x) = h^2 \sum_{y \in \Omega^{(h)}} G^{(h)}(x,y) \big( -\Delta^{(h)} f(y) \big) + \sum_{y \in \partial \Omega^{(h)}} G(x,y) f(y) \, .
\end{equation}
Combining this with the properties of the eigenfunctions/vectors, \cite[Section 5]{bramble1968effects} shows the following crucial identity: for $j \geq 1$,
\begin{equation}
	\label{eq:crucialBramble}
	(\mu^{(h)}_j - \mu) \big\langle \varphi_1 , \varphi^{(h)}_j \big\rangle_h = \big\langle \Delta \varphi_1 - \Delta^{(h)} \varphi_1 \, , \, \varphi^{(h)}_j \big\rangle_h = \mu^{(h)}_j \big\langle \Phi_h \, , \, \varphi^{(h)}_j \big\rangle_h \, ,
\end{equation}
where we have defined the function $\Phi_h$ by
\begin{equation*}
	\Phi_h(x) \defeq h^2 \sum_{y \in \Omega^{(h)}} G^{(h)}(x,y) \big[ \Delta \varphi_1(y) - \Delta^{(h)} \varphi_1(y) \big] \, .
\end{equation*}

Now, note that since $(\varphi^{(h)}_j)_{j \geq 1}$ is an orthonormal basis of $L^2(\Omega^{(h)})$, we get
\begin{equation}
	\label{eq:reste-parseval}
	\Big| \| \varphi_1 \|_{L^2,h}^2 - \big\langle \varphi_1, \varphi_1^{(h)} \big\rangle_h^2 \Big| = \sum_{j = 2}^{|\Omega^{(h)}|} \big\langle \varphi_1, \varphi^{(h)}_j \big\rangle_h^2 = \sum_{j = 2}^{|\Omega^{(h)}|} \bigg( \frac{\mu^{(h)}_j}{\mu^{(h)}_j - \mu} \bigg)^2 \big\langle \Phi_h , \varphi^{(h)}_j \big\rangle_h^2 \, ,
\end{equation}
where we used~\eqref{eq:crucialBramble} for the second identity.
Note that for any $j \geq 2$, $\mu^{(h)}_j \geq \mu^{(h)}_2$ (the eigenvalues are ordered) and that since $\mu_1$ is simple, the convergence of eigenvalues implies that $\mu^{(h)}_2 \to \mu_2 > \mu_1$. 
In particular, provided $h$ small enough, the ratio of eigenvalues in \eqref{eq:reste-parseval} is non-increasing in $j \geq 2$, and therefore is bounded by a constant independent of $h$ and $j$.
Thus, we have
\begin{equation*}
	\big| 1 - \big\langle \varphi_1, \varphi_1^{(h)} \big\rangle_h^2 \big| \leq \big| 1- \| \varphi_1 \|_{L^2,h}^2\big| +  C \sum_{j \geq  2} \big\langle \Phi_h , \varphi^{(h)}_j \big\rangle_h^2 \leq C \big( h^{p} + \| \Phi_h \|_{L^2,h}^2 \big) \, ,
\end{equation*}
using again that $ 1-\| \varphi_1 \|_{L^2,h}^2 = \grdO(h^p)$ as noticed above and the orthonormality of $(\varphi^{(h)}_j)_{j \geq 1}$.

All that is left is to prove that $\| \Phi_h \|_{2,h}^2 \leq C' h^p$.
Using a Taylor expansion, we easily see that $\Delta \varphi_1(y) - \Delta^{(h)} \varphi_1(y)$ regroups all the even orders of the expansion of $\varphi_1$ in a $h$-neighborhood of~$x$.
Stopping the Taylor expansion at the $4$-th order, we get
\begin{equation*}
	\Delta \varphi_1(y) - \Delta^{(h)} \varphi_1(y) = h^{-2} \sum_{\alpha} \Big[ \frac{D^\alpha \varphi_1 (y)}{\alpha !} + \eps_\alpha(h) \Big] \, ,
\end{equation*}
where the sum ranges over all multi-indices $\alpha = (\alpha_1, \dots , \alpha_4) \in \NN^4$ with $\alpha_i \geq 0, \alpha_1 + \dots + \alpha_4 = 4$, and where $\eps_\alpha(h) \to 0$ as $h \downarrow 0$ is an error term.
Therefore, Theorem~\ref{th:reg-cont-cone} imply that
\begin{equation}
	\label{eq:local-error-LP-int}
	\quad \big| \Delta \varphi_1(y) - \Delta^{(h)} \varphi_1(y) \big| \leq K h^2 d(y, \partial \Omega)^{p - 4} \,,\qquad  \forall y \in \Omega^{(h)}_{\geq 4} \,,
\end{equation}
for some constant $K > 0$.
Here, we have denoted $\Omega^{(h)}_{\geq 4}$ for the set of points of $\Omega^{(h)}$ at (graph) distance at least \(4h\) from the boundary; also, $p=1$ under Assumption~\ref{hyp:D} or $p\in (0,1)$ is given in~\eqref{def:p} under Assumption~\ref{hyp:cone}.
On the other hand, for the $y$'s near the boundary, we simply use Propositions~\ref{prop:unifbound} and~\ref{prop:unifcone} to get
\begin{equation}\label{eq:local-error-LP-bord}
	\big| \Delta \varphi_1(y) - \Delta^{(h)} \varphi_1(y) \big| \leq (2d + 1 + \mu_1) \sup_{z \in  \Omega^{(h)}_{<4}} |\varphi_1(z)| \leq K' h^p \,.
\end{equation}
where $\Omega^{(h)}_{<4} \defeq \Omega^{(h)} \setminus \Omega^{(h)}_{\geq 4}$.
Combining \eqref{eq:local-error-LP-int} and \eqref{eq:local-error-LP-bord}, we get
\begin{equation*}
	|\Phi_h(x)| \leq h^2 \sum_{y \in \Omega^{(h)}_{<4}} G^{(h)}(x,y) K h^2 d(y, \partial \Omega)^{p - 4} + K' h^p h^2 \sum_{y \in \Omega^{(h)} \setminus \Omega^{(h)}_{\geq 4}} G^{(h)}(x,y) \, .
\end{equation*}

We now use the following estimates on $G^{(h)}$: there are constants $c_1, c_2 > 0$ such that
\begin{equation}
	\sum_{y \in \Omega^{(h)} \setminus \Omega^{(h)}_{\geq 4}} G^{(h)}(x,y) \leq c_1 \,, \qquad h^{4} \sum_{y \in \Omega^{(h)}_{\geq 4}} G^{(h)}(x,y) d(y, \partial \Omega)^{p - 4} \leq c_2 h^{p} \, .
\end{equation}
These can easily be proven using probabilistic methods. 
For the first term, we see that the sum is the average number of visits to $\Omega^{(h)}_{< 4}$, which is dominated by a geometric random variable, hence the first term is bounded by a constant $c_1$. 
For the second term, we decompose the sum over $d(y,\partial \Omega) \in [2^k, 2^{k+1}] h$, so that the second term is bounded by a constant times
\[ 
h^{p} \sum_{k} 2^{-(4-p)k} G^{(h)}(x,A^{(h)}_k) \quad \text{with $A^{(h)}_k = \big\{ y \in \Omega^{(h)} \, : \, d(y,\partial \Omega^{(h)}) \in [2^k, 2^{k+1}] h \big\}$} \, . 
\]
We then observe that $G^{(h)}(x,A^{(h)}_k)$ is at most $c 2^{2k}$ for some $c > 0$, as this is the mean time a random walk stays inside a set with characteristic size $2^k$. 
Therefore, the sum over $k$ is bounded, leaving us with~$c_2 h^{p}$.

Combining all of the above, we therefore get $\sup_{x \in \Omega^{(h)}} |\Phi_h(x)| \leq c h^p$, which concludes the proof of Theorem~\ref{th:L2convergence}.
\qed

\subsection{Upgrading to a \texorpdfstring{$L^{\infty}$}{Linfty} convergence}
	
We now turn to the proof of Theorem~\ref{th:sup-convergence}, applying the ideas of \cite[\S7]{bramble1968effects}.
We use the Poisson formula \eqref{eq:poisson-formula} with $\varphi_1^{(h)}$ and $\varphi_1$ to get that for any \(x\in \Omega^{(h)}\)
\begin{align*}
	\varphi^{(h)}_1(x) - \varphi_1(x) &= h^2 \sum_{y \in \Omega^{(h)}} G^{(h)}(x,y) \big[ (-\Delta^{(h)} \varphi_1(y))  - (-\Delta^{(h)} \varphi_1^{(h)} (y)) \big] \\
	&= h^2 \sum_{y \in \Omega^{(h)}} G^{(h)}(x,y) \big[ \Delta \varphi_1(y) - \Delta^{(h)} \varphi_1(y)  + \Delta^{(h)} \varphi_1^{(h)}(y) - \Delta \varphi_1(y) \big] \\
	&= \Phi_h(x) + h^2 \sum_{y \in \Omega^{(h)}} G^{(h)}(x,y) \big[ \mu_1^{(h)} \varphi_1^{(h)}(y) - \mu_1 \varphi_1(y) \big] \\
	&= \Phi_h(x) + h^2 \sum_{y \in \Omega^{(h)}} G^{(h)}(x,y) \big[ (\mu_1^{(h)} - \mu_1) \varphi_1^{(h)}(y) + \mu_1 (\varphi_1^{(h)}(y) - \varphi_1(y) ) \big]\,.
\end{align*}	
We can now control the last sum using the eigenvalue and $L^2$ estimates derived above. 

We first note that according to Theorem~\ref{th:reg-cont-cone} we have that $\| \varphi^{(h)} \|_\infty$ is bounded by a constant, so
\begin{align*}
	\Big| (\mu_1^{(h)} - \mu_1) h^2 \sum_{y \in \Omega^{(h)}} G^{(h)}(x,y) \varphi_1^{(h)}(y) \Big| &\leq c |\mu_1^{(h)} - \mu_1| h^2 \sum_{y \in \Omega^{(h)}} G^{(h)}(x,y) \,.
\end{align*}
Then, we can write $h^2 \sum_{y \in \Omega^{(h)}} G^{(h)}(x,y) =h^2 \Ebf_x\big[ \tau^{(h)} \big]$ with $\tau^{(h)}$ the exit time of $\Omega^{(h)}$ by the simple random walk on $h\ZZ^d$, which is therefore bounded by a constant uniformly in $x \in \Omega^{(h)}$, by the invariance principle.
	
On the other hand, using Cauchy--Schwartz inequality, we have
\begin{equation*}
\Big| \mu_1 h^2 \sum_{y \in \Omega^{(h)}} G^{(h)}(x,y) \big( \varphi_1^{(h)}(y) - \varphi_1(y) \big) \Big| \leq \mu_1 h^2 \Big( \sum_{y \in \Omega^{(h)}} G^{(h)}(x,y)^2 \Big)^{1/2} \cdot \| \varphi_1^{(h)} - \varphi_1 \|_{L^2,h} \, .
\end{equation*}
In dimension $d \geq 3$, since $G^{(h)}(x,y)$ is bounded, the first factor above is bounded by a constant times $h^2 \Ebf_x\big[ \tau^{(h)} \big]^{1/2}$, which is itself bounded by a constant times $h$. In dimension~$2$, we instead have $G^{(h)}(x,y) \leq c \log h^{-1}$ (see \cite[Proposition 6.3.2]{lawlerRandomWalkModern2010}) which leads to a bound $c h^2 \Ebf_x\big[ \tau^{(h)} \big]^{1/2} \log h^{-1} \leq c' h \log h^{-1}$.
	
Combining the above, we find that for $h$ small enough,
\begin{equation*}
	\sup_{x \in \Omega^{(h)}} \big| \varphi^{(h)}_1(x) - \varphi_1(x) \big| \leq \sup_{x \in \Omega^{(h)}} \big| \Phi_h(x) \big| + c_1 |\mu_1^{(h)} - \mu_1| + c_2 h \log(h^{-1})^{\ind_{d = 2}} \| \varphi_1^{(h)} - \varphi_1 \|_{L^2,h} \, .
\end{equation*}
We have seen in the previous sections that all the terms above go to \(0\) as $h\downarrow 0$: more precisely, we have that 
$\sup_{x \in \Omega^{(h)}} \big| \Phi_h(x) \big| \leq C h^p$ and $\| \varphi_1^{(h)} - \varphi_1 \|_{2,h} \leq C h^{p/2}$. 
We therefore get that 
\[
	\sup_{x \in \Omega^{(h)}} \big| \varphi^{(h)}_1(x) - \varphi_1(x) \big| \leq C h^p + c_1 |\mu_1^{(h)} - \mu_1| \xrightarrow{h\downarrow0} 0 \,.
\]
In particular, under Assumption~\ref{hyp:D} we have $p=1$ and \(|\mu_1^{(h)} - \mu_1| \leq K h\) (recall~\eqref{eq:vitesse-cv-vp-pos-reach}), which gives the announced rate of decay.
Under Assumption~\ref{hyp:cone}, one needs a better control on $|\mu_1^{(h)} - \mu_1|$.

\end{appendix}

\subsection*{Acknowledgements.} 
The authors would like to thank Antoine Mouzard for numerous, enthusiastic, enlightening discussions.
We would also like to thank Karine Beauchard and Monique Dauge for helpful pointers to some of the literature.
The authors are also very much indebted to anonymous referees who pointed out several interesting references and helped improved the presentation of the paper.
Both authors also acknowledge the support of grant ANR Local (ANR-22-CE40-0012).

\printbibliography[heading=bibintoc]

\end{document}